\documentclass[11pt]{amsart}

\usepackage{amssymb}
\usepackage{amsmath}
\usepackage{enumerate}
\usepackage{verbatim}

\newcommand{\lang}[1]{\mathcal{L}(#1)}

\newcommand{\gh}{\mathcal{H}}
\newcommand{\gr}{\mathcal{R}}
\newcommand{\gl}{\mathcal{L}}
\newcommand{\gp}{\mathcal{P}}
\newcommand{\gd}{\mathcal{D}}
\newcommand{\gj}{\mathcal{J}}
\newcommand{\gq}{\mathcal{Q}}

\newcommand{\bbe}{\mathbb{E}}
\newcommand{\bbp}{\mathbb{P}}
\newcommand{\bbq}{\mathbb{Q}}
\newcommand{\bbr}{\mathbb{R}}
\newcommand{\bbs}{\mathbb{S}}
\newcommand{\bbf}{\mathbb{F}}

\newtheorem{thm}{Theorem}

\newtheorem{cor}{Corollary}
\newtheorem{defn}{Definition}

\newtheorem{lem}{Lemma}

 \theoremstyle{plain}

 \newtheorem{case}{Case}

\newcommand{\lb}{\langle}
\newcommand{\rb}{\rangle}

\usepackage{verbatim}
\usepackage{enumerate}

\begin{document}


\title[Homotopy Bases for Subgroups]{Homotopy Bases and Finite Derivation Type \\ for Subgroups of Monoids}
\maketitle

\begin{center}

    R. D. GRAY\footnote{The first author was supported by an EPSRC Postdoctoral Fellowship EP/E043194/1 held at the University of St Andrews, Scotland.}

    \medskip

    School of Mathematics, University of East Anglia \\ Norwich NR4 7TJ, U.K.

    \medskip

    \texttt{Robert.D.Gray@uea.ac.uk}

    \bigskip

    \bigskip

    A. MALHEIRO\footnote{This project was developed within the project PEst-OE/MAT/UI0143/2014 of CAUL, financed by FCT and FEDER.}

    \medskip

    Departamento de Matem\'{a}tica, Faculdade de Cincias e Tecnologia \\Universidade Nova de Lisboa, 2829--516 Caparica, Portugal. \\ Centro de \'{A}lgebra da Universidade de Lisboa \\ Av. Prof. Gama Pinto, 2, 1649--003 Lisboa, Portugal. 

    \medskip

    \texttt{ajm@fct.unl.pt} \\
\end{center}

\begin{abstract}
Given a monoid defined by a presentation, and a homotopy base for the derivation graph associated to the presentation, and given an arbitrary subgroup of the monoid, we give a homotopy base (and presentation) for the subgroup. If the monoid has finite derivation type ($\mathrm{FDT}$), and if under the action of the monoid on its subsets by right multiplication the strong orbit of the subgroup is finite, then we obtain a finite homotopy base for the subgroup, and hence the subgroup has $\mathrm{FDT}$. As an application we prove that a regular monoid with finitely many left and right ideals has $\mathrm{FDT}$ if and only if all of its maximal subgroups have $\mathrm{FDT}$. We use this to show that a finitely presented regular monoid with finitely many left and right ideals satisfies the homological finiteness condition $\mathrm{FP}_3$ if all of its maximal subgroups satisfy the condition $\mathrm{FP}_3$. 

\

\end{abstract}



\section{Introduction}

Geometric methods play an important role in combinatorial group theory (see \cite{LyndonandSchupp}). More recently, in \cite{Squier1} Squier (and in independent work Pride \cite{Pride2}, and Kilibarda \cite{Kilibarda1}) has developed a homotopy theory for monoids.  
Given a monoid presentation, Squier constructs a graph, called a \emph{derivation graph}, 
and then defines certain equivalence relations called \emph{homotopy relations} on the set of paths in the graph. A set of closed paths that generates the full homotopy relation is called a \emph{homotopy base}. A monoid defined by a finite presentation is said to have \emph{finite derivation type} ($\mathrm{FDT}$) if there is a finite homotopy base for the presentation. Squier showed that 
$\mathrm{FDT}$ is a property of the monoid, in the sense that it is independent of the choice of (finite) presentation. 

The homotopical finiteness condition $\mathrm{FDT}$ was originally introduced as a tool for the study of finite string rewriting systems. 
It was shown in \cite{Squier1} that if a monoid is defined by a finite complete rewriting system then that monoid has $\mathrm{FDT}$. In \cite{K2} Kobayashi showed that every one-relator monoid has $\mathrm{FDT}$; it is still an open question whether every one-relator monoid admits a finite complete rewriting system. 
More background on the relationship with finite complete rewriting systems may be found in \cite{OttoSurvey}. 

There are also important connections to the homology theory of monoids. If a monoid has $\mathrm{FDT}$ then it must satisfy the homological finiteness condition $\mathrm{FP}_3$, and for groups $\mathrm{FDT}$ and $\mathrm{FP}_3$ are equivalent; see \cite{Cremanns5}, \cite{Pride2}, \cite{Lafont} and \cite{Cremanns1}. In \cite{Pride4} it was shown that for monoids $\mathrm{FDT}$ and the homological finiteness condition $\mathrm{FHT}$ (in the sense of \cite{Wang8}) are not equivalent. In addition to this, Squier's homotopy theory has been applied in \cite{Guba} to the study of \emph{diagram groups}, which are precisely fundamental groups of Squier complexes of monoid presentations. 

One important area in the study of finiteness conditions is the consideration of their closure properties. Given a finiteness condition how does the property holding in a monoid relate to it holding in the substructures of that monoid (and vice versa)? For groups when passing to finite index subgroups, or finite index extensions, the properties of being finitely generated, finitely presented, having $\mathrm{FP}_n$, or having $\mathrm{FDT}$ are all preserved; see \cite{Magnus1}, \cite{Brown}. Results about the behaviour of $\mathrm{FDT}$ for monoids, when passing to submonoids and extensions appear in \cite{Wang1}, \cite{Malheiro1}, \cite{Pride4}, \cite{Malheiro3}, \cite{Wang7} and \cite{Malheiro2}.    

In this paper we investigate the relationship between the property $\mathrm{FDT}$ holding in a monoid, and the property holding in the subgroups of that monoid. Results of this type have been obtained for 
the properties of being finitely generated and presented (see \cite{Ruskuc2}), and for being residually finite (see \cite{Golubov1}). Specifically, given a monoid $S$ defined by a presentation $\mathcal{P}_S$, given an associated homotopy base $X$, and given a subgroup $G$ of $S$, we show how to obtain a homotopy base $Y$ for a presentation $\mathcal{P}_G$ of $G$ (this will be done in Section~3). Moreover, when the presentation $\mathcal{P}_S$ and the homotopy base $X$ are both finite, and $G$ has only finitely many \emph{cosets} in $S$ (where a coset of $G$ is a set $Gs$, with $s \in S$, such that there exists $s' \in S$ with $Gss'=G$) then the presentation $\mathcal{P}_G$ and homotopy base $Y$ will both be finite. A subgroup $G$ of a monoid $S$ with finitely many cosets, in the above sense, is said to have finite \emph{translational index}. This is equivalent to saying that $G$ acts on the left of its $\gr$-class with finite quotient. This notion of index was first considered in \cite{Ruskuc2} where it was shown that the properties of being finitely generated and presented are both inherited by subgroups with finite translational index (see also \cite{Steinberg1} for a quick topological proof of this result, in the special case of inverse semigroups). Note that in \cite{Ruskuc2} the author uses the term \emph{index} rather than translational index for the number of (right) cosets (in the above sense) of a subgroup of a monoid, but here we have opted for translational index to distinguish it from the various other notions of index for that have appeared in the literature; see \cite{Bergman}, \cite{Ruskuc&Thomas} and \cite{Gray1} for example. 

This gives our first main result.  

\begin{thm}\label{bigtosmall}
Let $S$ be a monoid and let $G$ be a subgroup of $S$. If $S$ has finite derivation type,
and $G$ has finite translational index in $S$, then $G$ has finite derivation type.
\end{thm}

Theorem~\ref{bigtosmall} proves one direction of our second main result.

\begin{thm}\label{main_0}
Let $S$ be a regular monoid with finitely many left and right ideals. Then $S$ has finite derivation type if and only if every maximal subgroup of $S$ has finite derivation type.
\end{thm}

Combining Theorem~\ref{main_0} with results of Cremanns and Otto \cite{Cremanns5}, \cite{Cremanns1} we obtain the following.

\begin{cor}\label{FP3}
Let $S$ be a finitely presented regular monoid with finitely many left and right ideals. If every maximal subgroup of $S$ satisfies the homological finiteness condition ${\rm FP}\sb 3$, then so does $S$.
\end{cor}
\begin{proof}
In \cite{Cremanns5} (see also \cite{Pride2} and \cite{Kilibarda1}) it was shown that if a monoid $S$ has finite derivation type then it also satisfies the homological finiteness condition ${\rm FP}_3$. Moreover, in \cite{Cremanns1} it was shown that for 
finitely presented groups having finite derivation type is equivalent to being of type ${\rm FP}_3$. 
Since $S$ is finitely presented with finitely many left and right ideals it follows from \cite{Ruskuc2} that all maximal subgroups of $S$ are finitely presented. 
Thus if every maximal subgroup of $S$ is of type ${\rm FP}_3$ then they all have finite derivation type, which by Theorem~\ref{main_0} implies that $S$ has finite derivation type, which in turn implies that $S$ is of type ${\rm FP}_3$. \end{proof}

These results answer several open problems posed in \cite[Open Problem~4.5]{Ruskuc2}. 
The converse of Corollary~\ref{FP3} does not hold; see \cite{Gray2}.

The paper is structured as follows. In Section~2 we give the preliminaries required for the rest of the paper, in particular we define what it means for a monoid to have finite derivation type. Then in Section~3 we prove our first main result which gives a homotopy base for an arbitrary subgroup of a monoid, and we prove Theorem~\ref{bigtosmall}. Finite derivation type for regular monoids and their maximal subgroups is the topic of Section~4, and this is the place where our second main result, Theorem~\ref{main_0}, is proved. 

\section{Preliminaries: homotopy relations and the homotopical finiteness property finite derivation type}

The background and results stated in Section~1 above have all been given in terms of monoids. In fact, all of the concepts that we intend to work with in the paper may be defined for semigroups and we shall find it convenient throughout to work in this slightly more general context. 

\subsection*{Presentations and rewriting systems}

Let $A$ be a non-empty set, that we call an \emph{alphabet}. We use $A^*$ to denote the free monoid over $A$, and $A^+$ to denote the free semigroup. For words $u,v \in A^*$ we write $u \equiv v$ to mean that $u$ and $v$ are identically equal as words. We write $|w|$ to denote the total number of letters in $w$, which we call the \emph{length} of the word $w$. A \emph{rewriting system} over $A$ is a subset $R \subseteq A^+ \times A^+$. An element of $R$ is called a \emph{rule}, and we often write $r_{+1}=r_{-1}$ for $(r_{+1}, r_{-1}) \in R$. For $u, v \in A^+$ we write $u \rightarrow_R v$ if $u \equiv w_1 r_{+1} w_2$, and $v \equiv w_1 r_{-1} w_2$ where $(r_{+1},r_{-1}) \in R$ and $w_1, w_2 \in A^*$. The reflexive symmetric transitive closure $\leftrightarrow_R^{*}$ of $\rightarrow_R$ is precisely the congruence on $A^+$ generated by $R$. The ordered pair $\langle A | R \rangle$ is called a \emph{semigroup presentation} with \emph{generators} $A$ and set of \emph{defining relations} $R$. If $S$ is a semigroup that is isomorphic to $A^+  / \leftrightarrow_R^{*}$ we say that $S$ is the \emph{semigroup defined by the presentation $\lb A | R \rb$}. A semigroup is said to be \emph{finitely presented} if it may be defined by a presentation with finitely many generators and a finite number of defining relations. For a word $w \in A^+$ we use $[w]_{\leftrightarrow_R^{*}}$ to denote the $\leftrightarrow_R^{*}$-class of $w$. So $[w]_{\leftrightarrow_R^{*}}$ is the element of $S = A^+ / \leftrightarrow_R^{*}$ that the word $w$ represents. For any subset $T$ of the semigroup $S$ we use $\lang{A,T}$ to denote the set of all words in $A^+$ representing elements of $T$, i.e.
$$
\lang{A,T} = \{ w \in A^+ : [w]_{\leftrightarrow_R^{*}}\; \in T \}.
$$

\subsection*{Derivation graphs, homotopy bases, and finite derivation type}

Let $\mathcal{P} = \lb A | R \rb$ be a semigroup presentation. The \emph{derivation graph} associated with $\gp$ is an infinite graph $\Gamma = \Gamma(\gp) = (V,E,\iota, \tau, ^{-1})$ with \emph{vertex set} $V = A^+$, and \emph{edge set $E$} consisting of the following collection of $4$-tuples:
\[
\{ (w_1, r, \epsilon, w_2): \ w_1, w_2 \in A^*, r=(r_{+1},r_{-1}) \in R, \ \mbox{and} \ \epsilon \in \{ +1, -1 \} \}.
\]
The functions $\iota, \tau : E \rightarrow V$ associate with each edge $\bbe = (w_1, r, \epsilon, w_2)$ (with $r=(r_{+1},r_{-1}) \in R$) its initial and terminal vertices $\iota \bbe = w_1 r_{\epsilon} w_2$ and $\tau \bbe = w_1 r_{- \epsilon} w_2$, respectively. The mapping $^{-1} : E \rightarrow E$ associates with each edge $\bbe = (w_1, r, \epsilon, w_2)$ an inverse edge $\bbe^{-1} = (w_1, r, -\epsilon, w_2)$.

A path is a sequence of edges $\bbp = \bbe_1 \circ \bbe_2 \circ \ldots \circ \bbe_n$ where $\tau \bbe_i  \equiv \iota \bbe_{i+1}$ for $i=1, \ldots, n-1$. Here $\bbp$ is a path from $\iota \bbe_1$ to $\tau \bbe_n$ and we extend the mappings $\iota$ and $\tau$ to paths by defining $\iota \bbp \equiv \iota \bbe_1$ and $\tau \bbp \equiv \tau \bbe_n$.  The inverse of a path $\bbp = \bbe_1 \circ \bbe_2 \circ \ldots \circ \bbe_n$ is the path $\bbp^{-1} = \bbe_n^{-1} \circ \bbe_{n-1}^{-1} \circ \ldots \circ \bbe_1^{-1}$, which is a path from $\tau \bbp$ to $\iota \bbp$. A \emph{closed path} is a path $\bbp$ satisfying $\iota \bbp \equiv \tau \bbp$. For two paths $\bbp$ and $\bbq$ with $\tau \bbp \equiv \iota \bbq$ the composition $\bbp \circ \bbq$ is defined. To unclutter notation slightly we shall usually omit the  symbol $\circ$ when composing paths, writing simply $\bbp \bbq$ in place of $\bbp \circ \bbq$.

We denote the set of paths in $\Gamma$ by $P(\Gamma)$, where for each vertex $w \in V$ we include a path $1_w$ with no edges, called the \emph{empty path} at $w$. We call a path $\bbp$ \emph{positive} if it is either empty or it contains only edges of the form $(w_1, r, +1, w_2)$. We use $P_+(\Gamma)$ to denote the set of all positive paths in $\Gamma$. Dually we have the notion of $\emph{negative}$ path, and $P_-(\Gamma)$ denotes the set of all negative paths. The free monoid $A^*$ acts on both sides of the set of edges $E$ of $\Gamma$ by
\[
x \cdot \bbe \cdot y = (x w_1, r, \epsilon, w_2 y)
\]
where $\bbe = (w_1, r, \epsilon, w_2)$ and $x, y \in A^*$. This extends naturally to a two-sided action of $A^*$ on $P(\Gamma)$ where for a path $\bbp = \bbe_1 \circ \bbe_2 \circ \ldots \circ \bbe_n$ we define
\[
x \cdot \bbp \cdot y = (x \cdot \bbe_1 \cdot y) \circ (x \cdot \bbe_2 \cdot y)
\circ \ldots \circ (x \cdot \bbe_n \cdot y).
\]
If $\bbp$ and $\bbq$ are paths such that $\iota \bbp \equiv \iota \bbq$ and $\tau \bbp \equiv \tau \bbq$ then we say that $\bbp$ and $\bbq$ are \emph{parallel}, and write $\bbp \parallel \bbq$. We use $\parallel\; \subseteq P(\Gamma) \times P(\Gamma)$ to denote the set of all parallel paths.

An equivalence relation $\sim$ on $P(\Gamma)$ is called a \emph{homotopy relation} if it is contained in $\parallel$ and satisfies the following conditions:

\begin{enumerate}[(H1)]

\item If $\bbe_1$ and $\bbe_2$ are edges of $\Gamma$, then
\[
(\bbe_1 \cdot \iota \bbe_2) (\tau \bbe_1 \cdot \bbe_2) \sim
(\iota \bbe_1 \cdot \bbe_2) (\bbe_1 \cdot \tau \bbe_2 ).
\]

\item For any $\bbp, \bbq \in P(\Gamma)$ and $x,y \in A^*$
\[
\bbp \sim \bbq \ \ \mbox{implies} \ \  x \cdot \bbp \cdot y \sim x \cdot \bbq \cdot y.
\]

\item For any $\bbp, \bbq, \bbr, \bbs \in P(\Gamma)$ with $\tau \bbr \equiv \iota \bbp \equiv \iota \bbq$ and $\iota \bbs \equiv \tau \bbp \equiv \tau \bbq$
\[
\bbp \sim \bbq \ \ \mbox{implies} \ \ \bbr\;\bbp\;\bbs \sim \bbr\;\bbq\;\bbs.
\]

\item If $\bbp \in P(\Gamma)$ then $\bbp \bbp^{-1} \sim 1_{\iota \bbp}$, where $1_{\iota \bbp}$ denotes the empty path at the vertex $\iota \bbp$.

\end{enumerate}

It is a straightforward exercise to check that the collection of all homotopy relations is closed under arbitrary intersection, and that $\parallel$ itself is a homotopy relation. Therefore, for any subset $C$ of $\parallel$ there is a unique smallest homotopy relation $\sim_C$ on $P(\Gamma)$ containing $C$. We call this the \emph{homotopy relation generated by $C$}. A subset $C$ of $\parallel$ that generates $\parallel$ is called a \emph{homotopy base for $\Gamma$}.

\begin{defn}
We say that the presentation $\gp = \lb A | R \rb$ has \emph{finite derivation type} (written $\mathrm{FDT}$ for short) if there is a finite homotopy base for $\Gamma = \Gamma(\gp)$. A finitely presented semigroup $S$ has finite derivation type if some (and hence any by \cite[Theorem~4.3]{Squier1} and \cite[Theorem~3]{Malheiro1})  finite presentation for $S$ has finite derivation type.
\end{defn}

Clearly, a set $B$ of parallel paths is a homotopy base if and only if the set $\{(\bbp \circ \bbq^{-1}, 1_{\iota \bbp}) : (\bbp, \bbq) \in B  \}$ is. Hence we say that a set $C$ of closed paths is a homotopy base if $\{ (\bbp, 1_{\iota \bbp}) : \bbp \in C\}$ is a homotopy base. So a homotopy base $X$ for $\Gamma = \Gamma(\mathcal{P})$ may be given either as a subset of $\parallel$, so that $X$ is a set of parallel paths, or may be given as a set of closed paths. Both of these ways of expressing homotopy bases will be used in this article. In addition, sometimes we shall refer to $X$ as a homotopy base of a presentation $\mathcal{P}$ (rather than of the graph $\Gamma(\mathcal{P})$) and also, when it is clear from context, as a homotopy base of the semigroup.    

The definition of $\mathrm{FDT}$ given above was first introduced for monoids by Squier in \cite{Squier1}. An equivalent geometric definition of $\mathrm{FDT}$, stated in terms of fundamental groups of $2$-complexes, may be found in \cite{Pride2}. Here, following \cite{Malheiro1}, we have opten to work with semigroups, semigroup presentations and the corresponding definition of $\mathrm{FDT}$. If one chose instead to work with monoid presentations, then the definitions given above would all have to be modified by replacing every occurrence of $A^+$ by $A^*$. Fortunately, when considering the properties of being finitely generated, presented or having $\mathrm{FDT}$ for monoids, it does not matter whether one chooses to work with semigroup presentations (and the corresponding definition of $\mathrm{FDT}$) or with monoid presentations (and corresponding definition of $\mathrm{FDT}$).   
Indeed, in \cite[Theorem~4.3]{Squier1} it was shown that for monoids the property of having $\mathrm{FDT}$ is independent of the choice of finite presentation for the monoid. The semigroup presentation analogue of this result was given in \cite[Theorem~3]{Malheiro1}. It is a very easy exercise to check that for a monoid, being finitely presented as a monoid is equivalent to being finitely presented as a semigroup. The analogue of this result for $\mathrm{FDT}$ is stated below; the proof is standard and so is omitted, full details may be found in \cite[Section~2.4]{MalheiroThesis}.

\begin{lem}
Let $M$ be a finitely presented monoid. Let $\mathcal{P}$ be a finite semigroup presentation defining $M$, and let $\mathcal{Q}$ be a finite monoid presentation defining $M$. Then $\mathcal{P}$ has finite derivation type if and only if $\mathcal{Q}$ has finite derivation type.
\end{lem}

Therefore when considering $\mathrm{FDT}$ for monoids, it makes no difference whether we work with semigroup presentations or with monoid presentations. Throughout the paper we shall work only with semigroup presentations. 

\section{A homotopy base for a subgroup of a semigroup}

In this section we present our first main result: Given a homotopy base for a semigroup we find a homotopy base for an arbitrary subgroup of that semigroup. When the semigroup has finite derivation type, and the subgroup has finite translational index, the obtained homotopy base for the subgroup will be finite, thus proving Theorem~\ref{bigtosmall}.

Let $S$ be a semigroup and let $G$ be a subgroup of $S$. As mentioned in the introduction, the \emph{right cosets} of $G$ are the elements of the strong orbit of $G$ under the action of $S$ on the set $\{ Gs : s \in S \}$. So, $Gs$ is a coset if and only if there exists $s' \in S$ with $Gss'=G$. By taking $s = s' = e$ where $e$ is the identity of $G$, which is an idempotent of $S$, 
we see that $G$ is a right coset of $G$. We call the number of right cosets the (right) \emph{translational index} of $G$ in $S$. We say that $G$ has finite right translational index if $G$ has finitely many right cosets. There is also an obvious dual notion of left coset, and left translational index. Throughout we shall work with right cosets, and by the translational index we shall always mean the right translational index. This definition of index for a subgroup of a semigroup was first considered in \cite{Ruskuc2}. These ideas have a natural interpretation in terms of Green's relations; see the end of this section for more details on this.

Let $S$ be defined by the semigroup presentation $\mathcal{P} = \lb A | R \rb$. The first thing that we need is a presentation for the group $G$. A method for finding such a presentation, given in terms of $\mathcal{P}$ and the action of $S$ on the cosets of $G$, was the main result of \cite{Ruskuc2}. We now give details of that result, which will be stated as Theorem~\ref{TPres}, before going on to consider the problem of homotopy bases.

Let $\mathcal{C} = \{ C_i : i \in I \}$ be the collection of (right) cosets of $G$ in $S$. Now $S$ acts on the set $\mathcal{C} \cup C_0$ in a natural way by right multiplication where $C_i s = C_0$ if and only if $C_i s \not\in \mathcal{C}$. This is the same as the action of $S$ on the index set $I \cup \{ 0 \}$ given by $C_i s = C_{is}$. We extend this notation to $A^+$ writing $i w \in I \cup \{0 \}$ to mean the element of $I \cup \{0 \}$ obtained by acting on $i$ by the element of $S$ that the word $w \in A^+$ represents.

By convention we let $1 \in I$ and set $C_1 = G$. For all $i \in I \setminus \{ 1 \}$ let $r_i, r_i'$ be fixed words in $A^+$, representing the elements $s_{r_i}$ and $s_{r_i'}$ of $S$ respectively, and chosen so that we have $G s_{r_i}  = C_i$ and $g s_{r_i} s_{r_i'} = g$ for all $g \in G$. (To show that such elements exist is an easy exercise; see \cite[Proposition~2.4]{Ruskuc2}.) For notational convenience we set $r_1 = r_1' = 1 \in A^*$ and also define $i 1 = i$ for all $i \in I$ (where $1$ denotes the empty word). Also let $e \in A^+$ be a fixed word representing the identity of the group $G$. Next define a new alphabet
\[
B = \{ [i,a]: i \in I, a \in A, ia \neq 0 \},
\]
and define a mapping
\[
\phi: \{(i,w): i \in I, w \in A^*, iw \neq 0    \} \rightarrow B^*
\]
inductively by
\[
\begin{array}{lll}
\phi(i,1) = 1, & &  \\
\phi(i, aw) = [i,a] \phi(ia, w) & & (i \in I, \ a \in A, \ w \in A^*, \ iaw \neq 0).
\end{array}
\]
This easily extends to
\begin{equation}
\phi(i,w_1 w_2) \equiv \phi(i, w_1)\phi(i w_1, w_2) \quad (i \in I, \ w_1, w_2 \in A^*, \ iw_1 w_2 \neq 0). \label{identity}
\end{equation}
For $w \in A^*$ such that $1w \neq 0$, abusing notation slightly, we write $\phi(w) = \phi(1,w)$ and we note that for any two words $w_3, w_3 \in \lang{A,G}$ we have
\[
\phi(1,w_3w_4) \equiv \phi(1,w_3) \phi(1w_3, w_4) \equiv \phi(1,w_3) \phi(1,w_4)
\]
so that
\begin{equation} \label{homonG}
\phi(w_3 w_4) \equiv \phi(w_3) \phi(w_4) \quad (\mbox{for all} \ w_3, w_4 \in \lang{A,G}).
\end{equation}
Note that even though $\phi$ is defined for the pairs $(i,1)$ $(i \in I)$, and its image includes the empty word, we have $\phi(i,w) \in B^+$ if and only if $w \in A^+$.  

With this notation the main theorem of \cite{Ruskuc2} states:
\begin{thm}\cite[Theorem~2.9]{Ruskuc2} \label{TPres}
If $S$ is defined by a presentation $\lb A | R \rb$ then, with the above notation, the presentation with generators $B$ and set of relations $U$:
\[
\begin{array}{lll}
 & \phi(i,u) = \phi(i,v) & (i \in I, (u=v) \in R, iu \neq 0), \\
	& [i,a] = \phi(er_iar_{ia}')  & (i \in I, \ a \in A, \ ia \neq 0)
\end{array}
\]
defines $G$ as a semigroup.
\end{thm}

Note that when $A$, $R$ and $I$ are finite, Theorem~\ref{TPres} gives a finite presentation for $G$.

Given a homotopy base $X$ for a presentation $\mathcal{P} = \lb A | R \rb$ of $S$ we shall now show how to obtain a homotopy base $Y$ for the presentation $\mathcal{Q} = \lb B | U \rb$ of $G$ given in Theorem~\ref{TPres}. When $A$, $R$, $I$ and $X$ are all finite, our homotopy base $Y$ (along with $B$ and $Q$) will be finite.

 The derivation graph $\Gamma(\mathcal{P})$ has vertices $A^+$ while the graph $\Gamma(\mathcal{Q})$ has vertices $B^+$. So $\phi$ gives a mapping from the vertices of $\Gamma(\mathcal{P})$ to the vertices of $\Gamma(\mathcal{Q})$. Next we let $\psi$ be the homomorphism extending
\[
\psi: B^* \rightarrow A^*, \quad [i,a] \mapsto e r_i a r_{ia}'.
\]
Note that in particular the empty word is mapped to the empty word under $\psi$,  which we write as $\psi(1) = 1$. It follows from the definition that for $u,v \in B^+$, if $u$ and $v$ represent the same element of $G$ then $\psi(u)$ and $\psi(v)$ represent that same element of $S$. Now by restricting to $B^+$, $\psi$ gives a mapping from the vertices of $\Gamma(\mathcal{Q})$ to the vertices of $\Gamma(\mathcal{P})$.    

In order to define a homotopy base for $\lb B | U \rb$ we shall carry out the following steps.
\begin{enumerate}[(I)]
\item Extend $\phi$ and $\psi$ to mappings between paths of $\Gamma(\mathcal{P})$ and paths of $\Gamma(\mathcal{Q})$.
\item For each $w \in B^+$ define a path $\Lambda_w$ in $\Gamma(\mathcal{Q})$ from $w$ to $\phi(\psi(w)) \in B^+$.
\end{enumerate}
In fact, $\phi$ will actually be defined on certain pairs $(i, \bbp)$ where $i \in I$ and $\bbp$ is a path from $\Gamma(\mathcal{P})$. Once these definitions have been made, we shall use them in Theorem~\ref{homotopybase} below to define a homotopy base for the presentation $\mathcal{Q}$ of $G$.

\subsection*{\boldmath Extending the definitions of $\phi$ and $\psi$ to mappings between paths} 

Now we shall extend $\phi$ to a mapping, also denoted $\phi$, with
\[
\phi: \{ (i,\bbp): i \in I, \ \bbp \in P(\Gamma(\mathcal{P})), \ i (\iota \bbp) \neq 0  \}
\rightarrow P(\Gamma(\mathcal{Q}))
\]
where for an edge $\bbe = (w_1, r, \epsilon, w_2)$ in $\Gamma(\mathcal{P})$ such that $i (\iota \bbe) \neq 0$ we define
\[
\phi(i, \mathbb{E}) =
(  \phi(i,w_1), \phi(i w_1, r_{+1}) = \phi(i w_1, r_{-1}), \epsilon, \phi(i w_1 r_{+1}, w_2)  )
\]
which is an edge of $\Gamma(\mathcal{Q})$, and then for a path $\bbp = \bbe_1 \bbe_2 \ldots \bbe_k$ in $\Gamma(\mathcal{P})$ we define
\[
\phi(i, \bbp) = \phi(i, \bbe_1) \phi(i, \bbe_2) \ldots \phi(i, \bbe_k).
\]
Also, for any word $w \in A^+$ and $i \in I$ such that $iw \neq 0$ we define $\phi(i,1_w) = 1_{\phi(i,w)}$. From the definition, and using \eqref{identity}, it is easy to see that $\iota \phi(i, \bbp) \equiv \phi(i, \iota \bbp)$, $\tau \phi(i, \bbp) \equiv \phi(i, \tau \bbp)$ and $\phi(i,\bbp)^{-1} = \phi(i,\bbp^{-1})$. As above we write $\phi(\bbp)$ as shorthand for $\phi(1,\bbp)$, for any paths $\bbp$ such that $1(\iota \bbp) \neq 0$. 

Immediately from this definition, combining with \eqref{identity} we obtain the identity
\begin{equation} \label{pathidentity}
\phi(i, w_1 \cdot \mathbb{P} \cdot w_2) =
\phi(i, w_1) \cdot \phi(i w_1, \mathbb{P}) \cdot \phi(i w_1 \iota \mathbb{P}, w_2)
\end{equation}
for any path $\mathbb{P}$ of $\Gamma(\mathcal{P})$ and $w_1, w_2 \in A^*$, $i \in I$ with $i w_1 \iota \mathbb{P} w_2 \neq 0$.

Next we extend $\psi: B^* \rightarrow A^*$ to a mapping, also denoted $\psi$, where
\[
\psi : P(\Gamma(\mathcal{Q})) \rightarrow P(\Gamma(\mathcal{P})).
\]
To do this, firstly for any word $w \in B^+$ we define $\psi(1_w) = 1_{\psi(w)}$. 
Next, for each $u \in U$ let $\bbe_u = (1,u,+1,1)$, which is an edge in $\Gamma(\mathcal{Q})$. Then for every $u \in U$ let $\psi(\bbe_u)$ be a fixed path in $\Gamma(\mathcal{P})$ from $\psi(u_{+1})$ to $\psi(u_{-1})$. Such a path exists since $\psi(u_{+1})$ and $\psi(u_{-1})$ represent the same element of $S$. Then for an arbitrary edge $\bbe = (w_1,u,\epsilon,w_2)$ of $\Gamma(\mathcal{Q})$ we define
\[
\psi(\bbe) = \psi(w_1) \cdot \psi(\bbe_u)^{\epsilon} \cdot \psi(w_2)
\]
which is a path from $\psi(\iota \bbe)$ to $\psi( \tau \bbe)$ since by definition $\psi : B^* \rightarrow A^*$ is a homomorphism. Then for any path $\bbp = \bbe_1 \bbe_2 \ldots \bbe_k$ in $\Gamma(\mathcal{Q})$ we define
\[
\psi(\bbp) = \psi(\bbe_1) \psi(\bbe_2) \ldots \psi(\bbe_k).
\]

\subsection*{\boldmath Definition of the paths $\Lambda_w$} For each $w \in B^+$ we shall define a path $\Lambda_w$ in $\Gamma(\mathcal{Q})$ from $w \in B^+$ to $\phi ( \psi (w) )$. We define 
$\Lambda_w$ by induction on the length of the word $w$. If $|w|=1$ then we have $w \equiv [i,a]$ for some $i \in I$, $a \in A$ such that $ia \neq 0$, and we set
\[
\Lambda_w = (1, [i,a] = \phi(er_i a r_{ia}'), +1, 1)
\]
which is an edge of $\Gamma(\mathcal{Q})$. If $|w|>1$ then write $w \equiv bw'$ where $b \in B$ and $w' \in B^+$, say $b = [i,a] \in B$. Then we define inductively
\[
\Lambda_w = (\Lambda_b \cdot w')(\phi(er_i a r_{ia}') \cdot \Lambda_{w'}).
\]
\begin{lem}
For all $w \in B^+$, $\Lambda_w$ is a path in $\Gamma(\mathcal{Q})$ with initial vertex $w$ and terminal vertex $\phi (\psi(w))$.
\end{lem}
\begin{proof}
The result holds trivially for words of length $1$. Now let $w \in B^+$ with $|w| \geq 2$ and assume inductively that the result holds for all words of strictly smaller length than $w$. Write $w \equiv bw'$ where $b = [i,a] \in B$. By definition \[
\Lambda_w = (\Lambda_b \cdot w')(\phi(er_i a r_{ia}') \cdot \Lambda_{w'})
= (\Lambda_b \cdot w')(\phi(\psi(b)) \cdot \Lambda_{w'}).
\]
By induction $\Lambda_w'$ is a path in $\Gamma(\mathcal{Q})$ from $w'$ to $\phi (\psi (w'))$. Now $\Lambda_b \cdot w'$ has initial vertex $bw'$ and terminal vertex $\phi(\psi(b))w'$, while $\phi(\psi(b)) \cdot \Lambda_{w'}$ has initial vertex $\phi(\psi(b))w'$ and terminal vertex $\phi(\psi(b)) \phi (\psi (w'))$. Since $\psi(b)$ and $\psi(w')$ both belong to $\lang{A,G}$ it follows from \eqref{homonG} and the fact $\psi$ is a homomorphism that
\[
\phi(\psi(b)) \phi(\psi(w')) \equiv \phi(\psi(b) \psi (w')) \equiv \phi(\psi(w)),
\]
so $\Lambda_w$ is a path from $bw' \equiv w$ to $\phi (\psi (w))$. \end{proof}

We are now in a position to define the homotopy base for $G$.

\begin{thm} \label{homotopybase}
Let  $X$ be a homotopy base for a presentation $\lb A | R \rb$ defining a semigroup $S$, where $X$ is a set of closed paths. Then, with the above notation, the set of closed paths $K \cup W$ is a homotopy base for the presentation $\lb B | U \rb$ of the group $G$ where
\[
K = \{ \phi(j, \mathbb{P}) : \mathbb{P} \in {X}, \ j \in I, \ j (\iota \mathbb{P}) \neq 0 \}
\]
and
\[
W = \{ \ \bbe_r \Lambda_{r_{-1}}  \phi(\psi(\bbe_r))^{-1}  \Lambda_{r_{+1}}^{-1} \ : r \in U, \ i \in I \},
\]
with $\bbe_r = (1,r,+1,1)$. 
\end{thm}

Note that once proved, this result has Theorem~\ref{bigtosmall} as a corollary, since if $X$, $A$, $R$ and $I$ are all finite then $B$, $U$ and $K \cup W$ will be finite.

\

\noindent \textbf{Proving Theorem~\ref{homotopybase}.} The rest of this section will be devoted to proving Theorem~\ref{homotopybase}.
We begin by writing an infinite homotopy base $Z$ for $\Gamma(\mathcal{Q})$ and then show that $K \cup W$ generates all the closed paths of $Z$ and hence is itself a homotopy base. Let $Z$ denote the following infinite set of closed paths of the graph $\Gamma(\mathcal{Q})$
\begin{align}
\mathbb{E} \ \Lambda_{\tau \mathbb{E}} \ (\phi ( \psi (\mathbb{E})))^{-1} \ {\Lambda^{-1}_{\iota \mathbb{E}}}, \ \mbox{for} \ \mathbb{E} \ \mbox{in} \ \Gamma(\gq) & & \ \  \label{z1} \\
\phi([\mathbb{E}_1,\mathbb{E}_2]), \ \mbox{for} \  \mathbb{E}_1, \mathbb{E}_2 \ \mbox{in} \  \Gamma(\gp) \ \mbox{such that} \  \iota \mathbb{E}_1 \iota \mathbb{E}_2 \in \lang{A,G}, & &  \  \ \label{z2} \\
 \quad \mbox{where $[\mathbb{E}_1, \mathbb{E}_2]$ is the path $(\mathbb{E}_1 \cdot \iota \mathbb{E}_2)
(\tau \mathbb{E}_1 \cdot \mathbb{E}_2)
(\mathbb{E}_1 \cdot \tau \mathbb{E}_2)^{-1}
(\iota \mathbb{E}_1 \cdot \mathbb{E}_2)^{-1}$} & & \notag \\
\phi(w_1 \cdot \mathbb{P} \cdot w_2), \ \mbox{for} \  \mathbb{P} \in {X} \ \mbox{and} \  w_1, w_2 \in A^* \ \mbox{such that} \ w_1 (\iota \mathbb{P}) w_2 \in \lang{A,G}. & & \  \ \label{z3}
\end{align}
\begin{lem}
The set $Z$ is a homotopy base for $\Gamma(\mathcal{Q})$.
\end{lem}
\begin{proof}
This actually follows from a more general result proved in \cite{Malheiro2}. Since in this case it is not difficult, and for the sake of completeness, we offer a sketch of the proof here. For full details we refer the reader to \cite{Malheiro2}.

Let $C$ be an arbitrary closed path in $\Gamma(\mathcal{Q})$. We must prove that $C \sim_Z 1_{\iota C}$.

Let $D = \psi(C)$, noting that $D$ is a closed path in $\Gamma(\mathcal{P})$ and that each vertex of $D$ belongs to $\lang{A,G}$. First we consider the closed path $\phi(\psi(C)) = \phi(D)$ in $\Gamma(\mathcal{Q})$. Since $X$ is a homotopy base for the presentation $\mathcal{P}$ it follows that in $\Gamma(\mathcal{P})$ we have $D \sim_X 1_{\iota D}$. This means (see \cite[Lemma~2.1]{K2}) that there is a finite sequence of applications of (H1)--(H4) that can be used to transform $D$ into $1_{\iota D}$. Using \eqref{z2} each application of (H1) can be turned into a corresponding application in    $\Gamma(\mathcal{Q})$, and using \eqref{z3} and the fact that the vertices of $D$ belong to $\lang{A,G}$, each application of (H2) can be turned into a corresponding application in $\Gamma(\mathcal{Q})$. The relations (H3) may be dealt with since by definition $\phi(\bbp\; \bbq\; \bbr) = \phi(\bbp) \phi(\bbq) \phi(\bbr)$, and the relations (H4) are handled using $\phi(\bbp^{-1}) = \phi(\bbp)^{-1}$. We conclude that  $\phi(\psi(C)) = \phi(D)$ is a closed path in $\Gamma(\mathcal{Q})$ with $\phi(D) \sim_Z 1_{\iota \phi(D)}$.

Now we return our attention to the original closed path $C$. Using the paths \eqref{z1}, applied to each edge of the closed path $C$ in turn, we see that $C \sim_Z  \Lambda_{\iota C}\; \phi(\psi(C)) \; \Lambda_{\tau C}^{-1}$ where $\iota C \equiv \tau C$ since $C$ is a closed path. We now have
\[
C \sim_Z  \Lambda_{\iota C}\; \phi(\psi(C)) \; \Lambda_{\tau C}^{-1}  \sim_Z \Lambda_{\iota C}\; 1_{\phi(\psi(C))} \; \Lambda_{\tau C}^{-1} = \Lambda_{\iota C} \Lambda_{\iota C}^{-1} \sim 1_{\iota C}
\]
as required. \end{proof}

It follows that to prove Theorem~\ref{homotopybase} it is sufficient to show the following.

\begin{lem}\label{KcupW}
The homotopy relation generated by $K \cup W$ contains $Z$, and hence $K \cup W$ is a homotopy base for $\Gamma(\mathcal{Q})$.
\end{lem}

The proof of Lemma~\ref{KcupW} will follow from the lemmas below. We consider each of the sets of closed paths \eqref{z1}, \eqref{z2}, and \eqref{z3} in turn.

\subsection*{Generating the paths \eqref{z2}}

\begin{lem}\label{type2}
Let $\bbe_1$ and $\bbe_2$ be edges from $\Gamma(\mathcal{P})$ where $\iota \bbe_1 \iota \bbe_2 \in \lang{A,G}$. Then  $\phi([\bbe_1, \bbe_2]) \sim 1_v$ in $\Gamma(\mathcal{Q})$,  where $v$ denotes the vertex 
$\iota \phi([\bbe_1, \bbe_2])$ of the graph $\Gamma(\mathcal{Q})$.
\end{lem}
\begin{proof}
First observe that by definition $\phi(1 \tau \mathbb{E}_1, \mathbb{E}_2)$ is a path from $\phi(1 \tau \mathbb{E}_1, \iota \mathbb{E}_2)$ to $\phi(1 \tau \mathbb{E}_1, \tau \mathbb{E}_2)$, and we have the equality
$
\phi(1 \iota \mathbb{E}_1, \iota \mathbb{E}_2) \equiv
\phi(1 \tau \mathbb{E}_1, \iota \mathbb{E}_2)
$
since $\iota \mathbb{E}_1$ and $\tau \mathbb{E}_1$ both represent the same element of $S$. Also the paths
$\phi(1 \tau \mathbb{E}_1, \mathbb{E}_2)$ and
$\phi(1 \iota \mathbb{E}_1, \mathbb{E}_2)$ are identically equal, since $1 \tau \mathbb{E}_1 = 1 \iota \mathbb{E}_1 \in I$.
Now we have
\[
\begin{array}{llll}
 &   & \phi( \bbe_1 \cdot \iota \bbe_2)  \phi ( \tau \bbe_1  \cdot \bbe_2 ) &   \\
& =     &       (\phi( \bbe_1 ) \cdot \phi( 1 \iota \bbe_1, \iota \bbe_2))
                (\phi( \tau \bbe_1) \cdot \phi( 1 \tau \bbe_1, \bbe_2 )) & \mbox{(by \eqref{pathidentity})} \\
& \sim  &       (\phi( \iota \bbe_1 ) \cdot \phi( 1 \tau \bbe_1, \bbe_2))
                (\phi( \bbe_1) \cdot \phi( 1 \tau \bbe_1, \tau \bbe_2 )) & \mbox{(applying (H1)} \\
& =     &        (\phi( \iota \bbe_1 ) \cdot \phi( 1 \iota \bbe_1, \bbe_2))
                (\phi( \bbe_1) \cdot \phi( 1 \iota \bbe_1, \tau \bbe_2 )) & \mbox{(since $1 \tau \mathbb{E}_1 = 1 \iota \mathbb{E}_1$)} \\
& = &           \phi( \iota \bbe_1 \cdot \bbe_2)  \phi ( \bbe_1  \cdot \tau \bbe_2 ) & \mbox{(by \eqref{pathidentity}).}
\end{array}
\]
Along with the definition of $\phi$ on paths and the fact that $\phi(\bbe)^{-1} = \phi(\bbe^{-1})$ this completes the proof of the lemma.
\end{proof}

\subsection*{Generating the paths \eqref{z3}}

\begin{lem}\label{type3}
For any $\mathbb{P} \in {X}$, and $w_1, w_2 \in A^*$ such that
$w_1 (\iota \mathbb{P}) w_2 \in \lang{A,G}$ we have  $\phi(w_1 \cdot \mathbb{P} \cdot w_2) \sim_K 1_v$, where $v$ denotes the vertex $\iota \phi(w_1 \cdot \mathbb{P} \cdot w_2)$ of the graph $\Gamma(\mathcal{Q})$.
\end{lem}
\begin{proof}
Since $\phi(1 w_1, \mathbb{P}) \in {K}$, using \eqref{pathidentity} we obtain
\[
\phi(w_1 \cdot \mathbb{P} \cdot w_2) = \phi(1, w_1) \cdot \phi(1 w_1, \mathbb{P}) \cdot \phi(1 w_1 \iota \mathbb{P}, w_2) \sim_{{K}} 1_v.
\]
\end{proof}

\subsection*{Generating the paths \eqref{z1}}

For this family of paths we must first prove two straightforward lemmas.

\begin{lem}\label{triangle}
For any edge $\bbe = (w_1, r, \epsilon, w_2)$ of $\Gamma(\mathcal{Q})$, with $\bbe_r = (1,r,+1,1)$
 we have
\[
\phi(\psi(\bbe)) = \phi(\psi(w_1)) \cdot \phi(\psi(\bbe_r)^{\epsilon}) \cdot \phi(\psi(w_2)).
\]
\end{lem}
\begin{proof}
It follows from the definition of $\psi(\bbe)$ and \eqref{pathidentity} that
\begin{align*}
\phi(\psi(\bbe))  & =  \phi(\psi(w_1) \cdot \psi(\bbe_r)^{\epsilon} \cdot \psi(w_2) ) \\
                  & =  \phi(1,\psi(w_1)) \cdot \phi(1 \psi(w_1), \psi(\bbe_r)^{\epsilon}) \cdot
                       \phi(1\psi(w_1) \iota \psi(\bbe_r)^{\epsilon}, \psi(w_2) ) \\
                 & =  \phi(\psi(w_1)) \cdot \phi(\psi(\bbe_r)^{\epsilon}) \cdot \phi(\psi(w_2)),
\end{align*}
since $\psi(w_1),  \psi(w_1) \iota \psi(\bbe_r)^{\epsilon} \in \lang{A,G}$ implies that $1 \psi(w_1) = 1 \psi(w_1) \iota \psi(\bbe_r)^{\epsilon} = 1$.
\end{proof}

\begin{lem}\label{split}
For all $w_1, w_2, w_3 \in B^+$ we have
\begin{enumerate}[(i)]
\item $\Lambda_{w_1 w_2} = (\Lambda_{w_1} \cdot w_2) (\phi (\psi (w_1)) \cdot \Lambda_{w_2} )$
\item $\Lambda_{w_1 w_2 w_3} = (\Lambda_{w_1} \cdot w_2 w_3) (\phi (\psi (w_1)) \cdot  \Lambda_{w_2} \cdot w_3) (\phi (\psi (w_1)) \phi (\psi (w_2)) \cdot \Lambda_{w_3}     )$.
\end{enumerate}
\end{lem}
\begin{proof}
Part (i) follows straight from the definition of $\Lambda_w$, and part (ii) is proved by applying part (i) twice.
\end{proof}

Now we can finish the proof.

\begin{lem}\label{type1}
For any edge $\mathbb{E}$ of $\Gamma(\mathcal{Q})$ we have
\[
\Lambda_{\iota \bbe} \; \phi (\psi (\bbe)) \sim_W \bbe \; \Lambda_{\tau \bbe}.
\]
\end{lem}
\begin{proof}
Let $\mathbb{E} = (w_1, r, \epsilon, w_2)$ be an arbitrary edge of $\Gamma(\mathcal{Q})$. First suppose that $w_1$ and $w_2$ are both non-empty. Then we have
\[
\begin{array}{llll}
\Lambda_{\iota \bbe} \phi (\psi (\bbe)) & = & (\Lambda_{w_1}\cdot r_{\epsilon} w_2)
(\phi (\psi(w_1)) \cdot \Lambda_{r_{\epsilon}} \cdot w_2)
(\phi (\psi (w_1)) \phi (\psi (r_{\epsilon})) \cdot \Lambda_{w_2})
 & \\
 & &  \quad \quad \quad \quad  \quad \quad  \quad \quad \quad \quad \quad \quad \phi(\psi(w_1)) \cdot \phi(\psi(\bbe_r)^{\epsilon}) \cdot \phi(\psi(w_2)) & \\
 & & &  \\
 & & \quad \quad \quad \quad   \quad \quad \quad  \quad \quad \quad \quad \quad \quad \mbox{(by Lemmas~\ref{triangle} and \ref{split})} & \\
 & & &  \\
 & \sim & (\Lambda_{w_1}\cdot r_{\epsilon} w_2)
(\phi (\psi(w_1)) \cdot \Lambda_{r_{\epsilon}} \cdot w_2)
(\phi (\psi (w_1)) \cdot \phi (\psi (\bbe_r)^{\epsilon}) \cdot w_2)
 & \\
 & &  \quad \quad \quad \quad  \quad \quad  \quad \quad \quad \quad \quad \quad \phi(\psi(w_1)) \phi(\psi(r_{-\epsilon})) \cdot \Lambda_{w_2} & \\
 & & &  \\
 & & \quad \quad \quad \quad   \quad \quad \quad  \quad \quad \quad \quad \quad \quad \mbox{(applying (H1))} & \\
 & & &  \\
 & \sim_W & (\Lambda_{w_1}\cdot r_{\epsilon} w_2)
(\phi (\psi(w_1)) \cdot \bbe_r^{\epsilon} \cdot w_2)
(\phi (\psi (w_1)) \cdot \Lambda_{r_{- \epsilon}} \cdot w_2)
 & \\
 & &  \quad \quad \quad \quad  \quad \quad  \quad \quad \quad \quad \quad \quad \phi(\psi(w_1)) \phi(\psi(r_{-\epsilon})) \cdot \Lambda_{w_2} & \\
 & & &  \\
 & & \quad \quad \quad \quad   \quad \quad \quad  \quad \quad \quad \quad \quad \quad \mbox{(since $\Lambda_{r_{\epsilon}} \phi (\psi (\bbe_r))^{\epsilon} \sim_W \bbe_r^{\epsilon} \Lambda_{r_{- \epsilon}} $)} & \\
 & & &  \\
 & \sim & (w_1 \cdot \bbe_r^{\epsilon} \cdot w_2)
(\Lambda_{w_1} \cdot r_{-\epsilon} w_2)
(\phi (\psi (w_1)) \cdot \Lambda_{r_{- \epsilon}} \cdot w_2)
 & \\
 & &  \quad \quad \quad \quad  \quad \quad  \quad \quad \quad \quad \quad \quad \phi(\psi(w_1)) \phi(\psi(r_{-\epsilon})) \cdot \Lambda_{w_2} & \\
 & & &  \\
 & & \quad \quad \quad \quad   \quad \quad \quad  \quad \quad \quad \quad \quad \quad \mbox{(applying (H1))} & \\
 & & &  \\
 & = & \bbe \Lambda_{\tau \bbe} \quad \quad \quad \quad \quad \ \  \quad \quad \quad \quad \quad \mbox{(by Lemma~\ref{split} and the definition of $\bbe$)}.
\end{array}
\]
If $w_1$ is empty, then this sequence of deductions holds if we replace every expression of the form $\Lambda_{w_1} \cdot u$ $(u \in B^+)$ by the empty path $1_u$ at $u$. Similarly, if $w_2$ is empty then the deductions above are valid if we replace each occurrence of $u \cdot \Lambda_{w_2}$ $(u \in B^+)$ by $1_u$.  
\end{proof}

This completes the proof of Lemma~\ref{KcupW} and hence also of
Theorem~\ref{homotopybase}. Moreover, since when all of $X$, $A$, $R$ and $I$ are finite it implies that all of $B$, $U$, and $K \cup W$ are finite, we obtain our first main result, Theorem~\ref{bigtosmall}, as a corollary.

\begin{cor}\label{mainresultcorol}
Let $S$ be a semigroup and let $G$ be a subgroup of $S$. If $S$ has finite derivation type,
and $G$ has finite translational index in $S$, then $G$ has finite derivation type.
\end{cor}

\subsection{Cosets and Green's relations}

The notion of coset being used in this article is best thought of when viewed from the point of view of Green's relations on the semigroup. Green's relations are the equivalence relations $\gr$, $\gl$, $\gd$, $\gh$ and  $\gj$ defined on a semigroup $S$ by
\[
\begin{array}{c}
x \gr y \Leftrightarrow xS^1 = yS^1, \ x \gl y \Leftrightarrow S^1x = S^1y, \ x \gj y \Leftrightarrow S^1xS^1 = S^1yS^1 \\
\gd = \gr \circ \gl = \gr \circ \gl, \ \gh = \gr \cap \gl.
\end{array}
\]
Here $S^1$ denotes the semigroup $S$ with an identity element $1$ adjoined. Since their introduction in \cite{Green1}, Green's relations have played a fundamental role in the development of the structure theory of semigroups. For more background on Green's relations and their importance in semigroup theory we refer the reader to \cite{Howie1}. In particular, the maximal subgroups of a semigroup $S$ are precisely the $\gh$-classes of $S$ that contain idempotents. Hence Theorem~\ref{bigtosmall} has the following corollary relating  the property of $\mathrm{FDT}$ in a semigroup  with the property in its maximal subgroups.

\begin{cor}\label{maxsubgroupFDT}
Let $S$ be a semigroup  and let $H$ be a maximal subgroup of $S$. If $S$ has finite derivation type, and the $\gr$-class of $H$ contains only finitely many $\gh$-classes, then $H$ has finite derivation type.
\end{cor}
\begin{proof}
The subgroup $H$ has finite translational index since its right cosets are precisely the $\gh$-classes of $S$ contained in the $\gr$-class of $H$.
\end{proof}

\section{Regular semigroups and finite derivation type}

In this section we turn our attention to regular semigroups. For a full account of regular semigroups and the other standard semigroup theoretic concepts mentioned in this section, such as completely $0$-simple semigroups, we refer the reader to \cite{Howie1}.

A semigroup  $S$ is regular if for all $x \in S$ there exists $y \in S$ such that $xyx=x$. This is equivalent to saying that every $\gr$-class and every $\gl$-class of $S$ contains an idempotent, and hence contains a maximal subgroup. In this sense we may think of a regular semigroup  as having ``lots'' of maximal subgroups and as a consequence it is often the case that the behaviour of such semigroups is closely linked to the behaviour of their maximal subgroups. Examples of this kind of result include \cite{Ruskuc2} where finite generation and presentability are considered, and \cite{Golubov1} which is concerned with residual finiteness. Here we consider the property $\mathrm{FDT}$, and the relationship between $\mathrm{FDT}$ holding in a regular semigroup, and $\mathrm{FDT}$ holding in the maximal subgroups of the semigroup.

As a special case of Corollary~\ref{maxsubgroupFDT} we see that if a regular semigroup $S$ has finitely many left and right ideals (which is equivalent to having finitely many $\gr$- and $\gl$-classes) then $\mathrm{FDT}$ holding in $S$ implies $\mathrm{FDT}$ holding in all of the maximal subgroups of $S$. The main aim of this section is to prove that the converse of this result is also true. We shall prove the following.

\begin{thm}\label{smalltobig}
Let $S$ be a regular semigroup with finitely many left and right ideals. If every maximal subgroup of $S$ has $\mathrm{FDT}$ then $S$ has $\mathrm{FDT}$.
\end{thm}

Along with Corollary~\ref{mainresultcorol} this proves the second main result, Theorem~\ref{main_0}, of our paper.

Our approach to the proof of Theorem~\ref{smalltobig} requires us to introduce a little more theory.  A semigroup is said to be \emph{completely ($0$-)simple} if it is ($0$-)simple and has ($0$-)minimal left and right ideals. Recall that a \emph{completely $0$-simple semigroup} $S$ is isomorphic to a \emph{$0$-Rees matrix semigroup} $M^0[G,I,\Lambda,P]$, where $G$ is a group isomorphic to any (and hence all) maximal subgroups of $S$, $I$ is a set indexed by the set of all $0$-minimal right ideals of $S$, $\Lambda$ is a set indexed by the set of all $0$-minimal left ideals of $S$, and $P = (p_{\lambda i})$ is a regular $\Lambda \times I$ matrix with entries from $G \cup \{ 0 \}$. Multiplication in $M^0[G,I,\Lambda,P] = (I \times G \times \Lambda) \cup \{ 0 \}$ is given by
\[
\begin{array}{c}
(i,g,\lambda)(j,h,\mu) =
\begin{cases}
(i, gp_{\lambda j}h, \mu) & \mbox{if $p_{\lambda j} \neq 0$} \\
0 &  \mbox{if $p_{\lambda j} = 0$} \\
\end{cases} \\
0 (i,g,\lambda) = (i,g,\lambda)0 = 00 = 0.
\end{array}
\]
There is an analogous construction for completely simple semigroups, given by taking Rees matrix semigroups $M[G,I,\Lambda,P] = I \times G \times \Lambda$ over groups, where $P$ is a matrix with entries from $G$, and multiplication is given by $(i,g,\lambda)(j,h,\mu) = (i, gp_{\lambda j}h, \mu)$. 

Given semigroups $S$, $T$ and $U$, we say that $S$ is an \emph{ideal extension} of $T$ by $U$ if $T$ is isomorphic to an ideal of $S$ and the Rees quotient $S/T$ is isomorphic to $U$.

In order to prove Theorem~\ref{smalltobig} it will be sufficient to establish the following.

\begin{thm}\label{comp0simpExt}
Let $S$ be an ideal extension of $T$ by a completely $0$-simple semigroup $U = M^0[G;I,\Lambda;P]$ with $I$ and $\Lambda$ both finite. If $T$ has $\mathrm{FDT}$ and $G$ has $\mathrm{FDT}$ then $S$ has $\mathrm{FDT}$.
\end{thm}

Before embarking on a proof of of Theorem~\ref{comp0simpExt}, we now show how Theorem~\ref{smalltobig} may be deduced from it.

\begin{proof}[Proof of Theorem~\ref{smalltobig}]

We begin with an easy observation. Let $N$ be a finitely presented semigroup, and let $N^0$ denote $N$ with a zero element $0$ adjoined. We claim that if $N^0$ has $\mathrm{FDT}$ then $N$ has $\mathrm{FDT}$. Indeed, let $\mathcal{P}_1 = \lb A | R \rb$ be a finite  presentation for $N$. Then $\mathcal{P}_0 = \lb A, 0 | R, a0 = 0a = 00 = 0 \ (a \in A) \rb$ is a finite presentation for $N^0$. If $X$ is a finite homotopy base for $\mathcal{P}_0$ then the subset $X' \subseteq X$ defined by
$
X' = \{ \bbp \in X : \iota \bbp \in A^+ \}
$
is clearly a finite homotopy base for the presentation $\mathcal{P}_1$ of $N$. (This is really just a special case of the semigroup presentation analogue of the result proved in \cite{Pride5} stating that $\mathrm{FDT}$ is inherited by submonoids with ideal complement.)

Now we use Theorem~\ref{comp0simpExt} to prove Theorem~\ref{smalltobig}. From the assumption that $S$ has finitely many left and right ideals it follows (see for example \cite[Chapter~6]{CliffordAndPreston}) that in $S$ we have $\gj = \gd$. The proof of Theorem~\ref{comp0simpExt} now goes by induction on the number of $(\gj = \gd)$-classes of $S$. When $S$ has just one $\gj$-class it follows, since $S$ has finitely many $\gr$- and $\gl$-classes, that $S$ is isomorphic to a completely simple semigroup $M[G;I,\Lambda,P]$ with $I$ and $\Lambda$ finite, and where $G$ has $\mathrm{FDT}$ by assumption. By Theorem~\ref{comp0simpExt} this implies that $S^0$ has $\mathrm{FDT}$, since it is an ideal extension of the trivial semigroup $T = \{ 0 \}$ by the completely $0$-simple semigroup $S^0$ where $G$ has $\mathrm{FDT}$. Now by the observation given in the previous paragraph, since $S^0$ has $\mathrm{FDT}$ it follows that $S$ has $\mathrm{FDT}$.

Now suppose that $S$ has at least two $\gj$-classes. Let $J_M$ be a maximal $\gj$-class (in the natural ordering of $\gj$-classes $J_x \leq_{\gj} J_y \Leftrightarrow S^1 x S^1 \subseteq S^1 y S^1$). Then $T = S \setminus J_M$ is an ideal of $S$ where $T$ is regular and, since $T$ is a union of $(\gj = \gd)$-classes of $S$ each of which is regular and contains only finitely many $\gr$- and $\gl$-classes, it follows that  
$T$ has strictly fewer $(\gj = \gd)$-classes than $S$. So $T$ is a regular semigroup with finitely many left and right ideals, and every maximal subgroup of $T$ is a maximal subgroup of $S$ and thus has $\mathrm{FDT}$ by assumption. Hence by induction $T$ has $\mathrm{FDT}$. But now $S$ is an ideal extension of $T$ by the Rees quotient $S / T \cong M^0[G;I,\Lambda;P]$ where $I$ and $\Lambda$ are finite (since $S$ has finitely many left and right ideals), and $G$  has $\mathrm{FDT}$. Applying Theorem~\ref{comp0simpExt} we conclude that $S$ has $\mathrm{FDT}$.
\end{proof}

The rest of this section will be devoted to the proof of Theorem~\ref{comp0simpExt}.

\

\noindent \textbf{Proof of Theorem~\ref{comp0simpExt}.} Let $S$ be an ideal extension of a semigroup $T$ by the completely $0$-simple semigroup $U = M^0[G;I,\Lambda;P]$. So the semigroup $S$ decomposes as the disjoint union $S = T \cup (I \times G \times \Lambda)$. Without loss of generality we may suppose that $1 \in I$, $1 \in \Lambda$ and that $p_{1 1} \neq 0$ so that $\{ 1 \} \times G \times \{ 1 \}$ is a group $\gh$-class. Let $\mathcal{P}_T = \lb Z | Q \rb$ be a semigroup presentation for $T$, and let $\mathcal{P}_G = \lb A | R \rb$ be a semigroup presentation for the group $G$. Our first task is to write down a presentation for $S$ with respect to which we shall define our homotopy base. We obtain a presentation for $S$ by applying the results \cite[Theorem~6.2]{HowieAndRuskuc} and \cite[Proposition~4.4]{Ruskuc2} in the following way.

We start by fixing some notation. Let $e \in A^+$ is a fixed word representing the identity element of $G$. Let $B = \{ b_i : i \in I \setminus \{ 1 \} \}$ and $C = \{ c_{\lambda} : \lambda \in \Lambda \setminus \{ 1 \}  \}$. Ultimately in the presentation for $S$ we give below, $A \cup B \cup C$ will be a subset of the generators, where $b_i$ represents the element $(i,1,1)$, $c_{\lambda}$ represents the element $(1,1,\lambda)$, and $a \in A$ represents the element $(1,g_a,1)$ where $g_a \in G$ is the element of $G \cong \lb A  | R \rb$ represented by $a \in A$. Given this, for every word $u \in (A \cup B \cup C)^+$ representing the zero of $U = M^0[G; I, \Lambda;P]$ fix a word $\rho(u) \in Z^+$ such that the relation $u = \rho(u)$ holds in $S$. Similarly for every pair of letters $x \in A \cup B \cup C$ and $z \in Z$ fix words $\sigma(z,x), \tau(x,z) \in Z^+$ such that the relations $zx = \sigma(z,x)$ and $xz = \tau(x,z)$ hold in $S$.

Now, from \cite[Theorem~6.2]{HowieAndRuskuc} we can use the presentation $\lb A | R \rb$ and the matrix $P$ to obtain a presentation for $U = M^0[G;I,\Lambda;P]$. In \cite[Theorem~6.2]{HowieAndRuskuc} this is given as a presentation of a semigroup \emph{with zero} which we must convert into a genuine semigroup presentation by adding a generator $0$ and relations $0x=x0=00=0$ for all other generators. As well as this, we also add the following additional (redundant) relations:
\[
\begin{array}{llll}
a b_i = a p_{1 i} & (a \in A, \ i \in I), & c_{\lambda} a = p_{\lambda 1} a & (a \in A, \ \lambda \in \Lambda), \\
c_{\lambda} c_{\mu} = p_{\lambda 1} c_{\mu} & (\mu, \lambda \in \Lambda), & b_i b_j = b_i p_{1 j} & (i,j \in I).
\end{array}
\]
where the symbols $p_{\lambda i}$ appearing in the relations are really fixed words from $A^+$ representing the elements $p_{\lambda i }$ of $G$. These relations are all easily seen to be consequences of the relations appearing in the presentation of \cite[Theorem~6.2]{HowieAndRuskuc} and hence the resulting presentation also defines $U = M^0[G;I,\Lambda;P]$. Then taking this presentation for $U = M^0[G;I,\Lambda;P]$ together with the presentation $\lb Z | Q \rb$ of $T$, and applying \cite[Proposition~4.4]{Ruskuc2} we obtain the presentation for $S$ given below.
\[
\begin{array}{cc}
\begin{array}{rcl}
\mathcal{P}_S = \langle \ A, \ B, \ C, \ Z \ | \ R, \ Q, & &    \\ 
b_i e = b_i, \ ec_{\lambda} = c_{\lambda}, & &  \\
zx = \sigma(z,x), \ xz = \tau(x,z) & & (x \in A \cup B \cup C, z \in Z)  \\
eb_i = p_{1 i}, \ ab_i = ap_{1 i} & & (a \in A, i \in I: p_{1 i} \neq 0)    \\
c_{\lambda} e = p_{\lambda 1}, \ c_{\lambda} a = p_{\lambda 1} a & & (a \in A, \lambda \in \Lambda: p_{\lambda 1} \neq 0) \\
c_{\lambda} b_i = p_{\lambda i} & & (i \in I, \lambda \in \Lambda : p_{\lambda i} \neq 0) \\
c_{\lambda} c_{\mu} = p_{\lambda 1} c_{\mu} & & (\mu, \lambda \in \Lambda : p_{\lambda 1} \neq 0) \\
b_i b_j = b_i p_{1 j} & & (i, j \in I: p_{1 j} \neq 0) \\
e b_i = \rho(eb_i), \ ab_i = \rho(ab_i) & & (a \in A, i \in I: p_{1 i} = 0) \\
c_{\lambda} e = \rho(c_{\lambda} e), \ c_{\lambda} a = \rho(c_{\lambda} a) & & (a \in A, \lambda \in \Lambda: p_{\lambda 1} = 0) \\
c_{\lambda} b_i = \rho(c_{\lambda} b_i) & & (i \in I, \lambda \in \Lambda: p_{\lambda i}=0) \\
c_{\lambda} c_{\mu} = \rho(p_{\lambda 1} c_{\mu}) & & (\mu, \lambda \in \Lambda: p_{\lambda 1} = 0) \\
b_i b_j = \rho(b_i p_{i j}) & & (i, j \in I: p_{i j}=0) \ \rangle
\end{array} & 
\begin{array}{c}
\left. \begin{array}{l}
(4.1) \\
(4.2) \ \} \ R_e \\
(4.3) \ \} \ R_0
\end{array} \right. \\
\left. \begin{array}{r}
(4.4) \\
(4.5) \\
(4.6) \\
(4.7) \\
(4.8) 
\end{array} \right\} R_U \\
\left. \begin{array}{r}
(4.9) \\
(4.10) \\
(4.11) \\
(4.12) \\
(4.13) 
\end{array} \right\} R_T
\end{array}
\end{array}
\]
We group the relations of this presentation together into the following sets
\[
\begin{array}{c}
    \begin{array}{cc}
    R_e  =  (4.2), & R_0  =  (4.3)
    \end{array} \\
R_U =  (4.4) \cup (4.5) \cup (4.6) \cup (4.7) \cup (4.8) \\
R_T =  (4.9) \cup (4.10) \cup (4.11) \cup (4.12) \cup (4.13),
\end{array}
\]
so that
\[
\mathcal{P}_S = \langle \ A, \ B, \ C, \ Z \ |  \ R, \ R_e, \ R_0, \ R_U, \ R_T, \ Q \rangle.
\]
Let $\Gamma_G = \Gamma(\mathcal{P}_G)$, $\Gamma_T = \Gamma(\mathcal{P}_T)$ and $\Gamma_S = \Gamma(\mathcal{P}_S)$. In the natural way we view $\Gamma_G$ and $\Gamma_T$ as subgraphs of $\Gamma_S$. Let $\Gamma$ be the subgraph of $\Gamma_S$ with the same vertex set, but which only contains those edges $(w_1, r, \epsilon, w_2)$ with $r \in R_0 \cup R_U \cup R_T$. Recall that we use $P_+(\Gamma)$ to denote the set of all positive paths in $\Gamma$ (i.e. those paths that are either empty or have the property that $\epsilon = +1$ for every edge of the path). Let $\Gamma_U$ (respectively $\Gamma_0$) denote the subgraph of $\Gamma_S$ with the same vertex set as $\Gamma_S$, but which contains only those edges of the form $(w_1, r, \epsilon, w_2)$ where $r \in R_U$ (respectively $r \in R_0$).

Before we begin the process of constructing a homotopy base for $S$, we must prove some lemmas about the presentation $\mathcal{P}_S$. We write $B^1 A^* C^1$ to denote the set of all non-empty words from the set
\[
\{ bwc : b \in B \cup \{ 1 \}, \ c \in C \cup \{ 1 \} \ \& \ w \in A^*  \}.
\]
Likewise we use the notation $B^1 A^*$ and $A^* C^1$. We shall refer to the set of words $B^1 A^* C^1 \cup Z^+$ as the set of \emph{quasi-normal forms} for the presentation $\mathcal{P}_S$. 

\begin{lem}\label{dagger}
\ \

\begin{enumerate}[(i)]
\item The word $b_i w c_\lambda$ with $i \in I \setminus \{ 1 \}$, $\lambda \in \Lambda \setminus \{ 1 \}$ and $w \in A^*$ represents the element $(i,g,\lambda)$ of $S$, where $g \in G$ is the element represented by $w$ if $w \in A^+$, or $g=1$ if $w$ is the empty word.
\item The word $b_i w$ with $i \in I \setminus \{ 1 \}$ and $w \in A^*$ represents the element $(i,g,1)$ of $S$, where $g \in G$ is the element represented by $w$ if $w \in A^+$, or $g=1$ if $w$ is the empty word.
\item The word $w c_\lambda$ with $\lambda \in \Lambda \setminus \{ 1 \}$ and $w \in A^*$ represents the element $(1,g,\lambda)$ of $S$, where $g \in G$ is the element represented by $w$ if $w \in A^+$, or $g=1$ if $w$ is the empty word.
\item The word $w \in A^+$ represents the element $(1,g,1)$ where $g \in G$ is the element represented by the word $w$. 
\end{enumerate} 
\end{lem}
\begin{proof}
This follows from the construction of the presentation $\mathcal{P}_S$, together with the proof of \cite[Theorem~6.2]{HowieAndRuskuc}. 
\end{proof}

Given $b \in B^1$ and $c \in C^1$ we write $b \cdot P(\Gamma_G) \cdot c$ to denote the set of all paths $b \cdot \bbp \cdot c$ where $\bbp \in P(\Gamma_G)$, and we write 
$
B^1 \cdot P(\Gamma_G) \cdot C^1 = \bigcup_{b \in B^1, c \in C^1}{ b \cdot P(\Gamma_G) \cdot c}.
$
Since the vertices of the graph $\Gamma_G$ are the (non-empty) words $A^+$ it follows that the set of vertices appearing in paths from $B^1 \cdot P(\Gamma_G) \cdot C^1$ is $B^1 A^+ C^1$. So there are words in $B^1 A^* C^1$ that are \emph{not} vertices of paths in $B^1 \cdot P(\Gamma_G) \cdot C^1$, namely the words $B \cup C \cup BC$. We shall often be interested in paths with vertices from $B^1 A^* C^1$ and edges of the form $(w_1, r, \epsilon, w_2)$ where $r \in R \cup R_e$. So, given $b \in B^1$ and $c \in C^1$ we write $\overline{b \cdot P(\Gamma_G) \cdot c}$ to denote the set of all paths whose vertices all belong to $b A^* c \cup \{ bc \}$ (where $\{bc\} = \{b\}$ when $c=1$, and $\{bc\} = \{c\}$ when $b=1$) and whose edges are all either empty or of the form $(w_1, r, \epsilon, w_2)$ where $r \in R \cup R_e$. Clearly $b \cdot P(\Gamma_G) \cdot c \subseteq \overline{b \cdot P(\Gamma_G) \cdot c}$ for all $b \in B^1$, $c \in C^1$. We also write 
$
\overline{B^1 \cdot P(\Gamma_G) \cdot C^1} = \bigcup_{b \in B^1, c \in C^1}{ \overline{b \cdot P(\Gamma_G) \cdot c}}
$
noting that the set of vertices arising in paths of $\overline{B^1 \cdot P(\Gamma_G) \cdot C^1}$ is $B^1 A^* C^1$. 

The next lemma shows how using only positive relations from the set $R_0 \cup R_U \cup R_T$ we can transform an arbitrary word into a word in quasi-normal form.
\begin{lem}\label{Lemma1}
Let $w \in (A \cup B \cup C \cup Z)^+$ be arbitrary.
\begin{enumerate}[(i)]
\item If $w$ represents an element of $S \setminus T$ then there is a path $\bbp \in P_+(\Gamma)$ from $w$ to some $w' \in B^1 A^* C^1$. 
\item If $w$ represents an element of $T$ then there is a path $\bbp \in P_+(\Gamma)$ from $w$ to some $w' \in Z^+$. Moreover, if $w$ contains a letter from $Z$ then $\bbp$ may be chosen with $\bbp \in P_+(\Gamma_0)$. 
\end{enumerate}
\end{lem}
\begin{proof}
We prove (i) and (ii) simultaneously. The proof is by induction on the total number $|w|_{B \cup C}$ of letters in the word that come from the set $B \cup C$. If $|w|_{B \cup C}=0$ then either $w$ contains a letter from $Z$, in which case there is an obvious path in $P_+(\Gamma_0)$ from $w$ to a word $w' \in Z^+$, or $w \in A^+$, so  $w$ represents an element of $G$ and is already written in the required form, and we are done by setting $w' \equiv w$ and $\bbp$ to be the empty path.

Now suppose that $|w|_{B \cup C} > 0$. If $w \in B^1 A^* C^1 \cup Z^+$ then we are done by setting $w' \equiv w$, $\bbp$ to be the empty path, and using the fact that the words in $B^1 A^* C^1$ all represent elements of $S \setminus T$. Now suppose that $w \not\in B^1 A^* C^1 \cup Z^+$. Then by inspection of the relations in the presentation $\mathcal{P}_S$ it follows there exists a positive edge $\bbe \in P_+(\Gamma)$ with $\iota \bbe \equiv w$. Now the result follows by induction since by considering the relations $R_0 \cup R_U \cup R_T$ we see that with $v \equiv \tau \bbe$ we have $|v|_{B \cup C} < |w|_{B \cup C}$.

The last part of (ii) is an obvious consequence of the definition of $\Gamma_0$.
\end{proof}

The following technical lemma gives some information about what happens if we try to rewrite a word into quasi-normal form by first splitting the word into three parts, and then rewriting each of the parts in turn. 

\begin{lem}\label{trick}
Let $w_1, w_3 \in B^1 A^* C^1 \cup \{ 1 \}$ and let $w_2 \in (A \cup B \cup C \cup Z)^+$ be arbitrary. Write $w_1 \equiv w_1' \alpha$ and $w_3 \equiv \beta w_3'$ where $\alpha, \beta \in A \cup B \cup C \cup \{1 \}$, and where $\alpha$ (resp. $\beta$) is empty only if $w_1$ (resp. $w_3$) is empty. If $\alpha w_2 \beta$ represents an element in $S \setminus T$ and $u =\alpha w_2 \beta$ holds in $S$ with $u \in B^1 A^* C^1$, then
\[
w_1' u w_3' \in B^1 A^* C^1.
\]
In particular, $w_1 w_2 w_3 = w_1' u w_3'$ represents an element of $S \setminus T$. 
\end{lem}
\begin{proof}
There are several cases to consider.
\begin{case}
$w_1'$ and $w_3'$ are both empty. 
\end{case}
In this case $w_1' u w_3' \equiv u \in B^1 A^* C^1$ by assumption. 

\begin{case}
Neither $w_1'$ nor $w_3'$ is empty.
\end{case}
In this case it follows that $\alpha \in C \cup A$ and $\beta \in B \cup A$, since $w_1, w_3 \in B^1 A^* C^1$ where $w_1 \equiv w_1' \alpha$ and $w_3 \equiv \beta w_3'$. Now consider the element of the semigroup $S$ represented by the word $\alpha w_2 \beta$. By assumption  $\alpha w_2 \beta$ represents an element in $S \setminus T$. Moreover since $\alpha \in C \cup A$ it follows from Lemma~\ref{dagger} that the element that $\alpha$ represents is a triple $(i, g, \lambda) \in I \times G \times \Lambda$ with $i=1$. Similarly the letter $\beta$ represents a triple $(i,g,\lambda)$ with $\lambda = 1$. Therefore, since by assumption $\alpha w_2 \beta$ represents an element of $S \setminus T$ it follows from the definition of multiplication in a $0$-Rees matrix semigroup that it represents a triple of the form $(1, g, 1)$ for some $g \in G$. So by Lemma~\ref{dagger} since $u \in B^1 A^* C^1$ and $u$ also represents the element $(1,g,1)$ we conclude that in fact $u \in A^+$. In conclusion we have $w_1' \in B^1 A^*$, $u \in A^+$ and $w_3' \in A^* C^1$, and hence $ w_1' u w_3' \in B^1 A^* C^1$, as required. 

\begin{case}
$w_1'$ is non-empty and $w_3'$ is empty (and the dual of this case).  
\end{case}
We argue in a similar way to the previous case. Since $w_1'$ is non-empty we have $\alpha \in A \cup C$, and since $w_3'$ is empty it follows that $|w_3|\leq 1$. Now consider the word $\alpha w_2 \beta$. Since $\alpha \in A \cup C$ it follows that $\alpha$ represents an element of the form $(1,g,\lambda)$ and hence, by the definition of multiplication in a $0$-Rees matrix semigroup, since $\alpha w_2 \beta$ represents an element of $S \setminus T$, it must represent an element $(1,h,\mu)$ for some $h \in G$, $\mu \in \Lambda$. Therefore since $u \in B^1 A^* C^1$ with $u = \alpha w_2 \beta$ in $S$, it follows from Lemma~\ref{dagger} that $u \in A^* C^1$. Since $w_1'$ being non-empty implies that $\alpha$ is non-empty we see that $w_1' \in B^1 A^*$. Now since $w_3'$ is empty we have $w_1' u w_3' \equiv w_1' u \in B^1 A^* C^1$ as required.

The fact that $w_1 w_2 w_3$ represents an element of $S \setminus T$ follows from Lemma~\ref{dagger}.
\end{proof}

\subsection*{\boldmath A homotopy base for $\mathcal{P}_S$}

Now we begin the process of building a homotopy base for $\gp_S$. Given a homotopy base $X_G$ for the presentation $\mathcal{P}_G = \lb A | R \rb$, and a homotopy base $X_T$ for the presentation $\mathcal{P}_T = \lb Z | Q \rb$, we shall construct a homotopy base $X = X_1 \cup X_1' \cup X_2 \cup X_3 \cup X_e \cup X_G \cup X_T$ for the presentation $\mathcal{P}_S$ where $X_1$, $X_1'$, $X_2$, $X_3$ and $X_e$ are defined below. The proof that $X$ is a homotopy base will be given in Lemma~\ref{wrappingup}. In particular, when $X_G$ and $X_T$ are finite, the presentations $\mathcal{P}_G$ and $\mathcal{P}_T$ are finite, and the sets $I$ and $\Lambda$ are finite, then $P_S$ is a finite presentation, and it will follow from the definitions that $X$ is a finite homotopy base for the finite presentation $P_S$, and that will complete the proof of Theorem~\ref{comp0simpExt}.

We shall build a homotopy base $X$ for $\gp_S$ in stages. Our approach is inspired, in part, by methods used by Wang in \cite{Wang1}. 

\

\noindent \textbf{The parallel paths $\mathbf X_1$ and $\mathbf X_1'$.} For every pair $\beta_1, \beta_2 \in Z^+$ such that $\beta_1 = \beta_2$ in $T$, let $\mathbb{P}_{\beta_1,\beta_2}$ be a fixed path in $\Gamma_T$ from $\beta_1$ to $\beta_2$, where these paths are chosen in such a way that $\bbp_{\beta_2, \beta_1} = \bbp_{\beta_1, \beta_2}^{-1}$, for all $\beta_1, \beta_2 \in Z^+$. Let $X_1$ be the set of parallel paths of $\Gamma_S$ of the form:
\begin{align*}
( \ (x, l=r, +1, 1),
(1, xz = \tau(x,z), +1, l')
\quad \quad \quad \quad \quad \quad \quad \quad \quad \quad \quad \quad \quad \quad
\\
\quad \quad \quad \quad \quad \quad \quad \quad \quad \quad \quad \quad \quad \quad
(\bbp_{\tau(x,z)l', \tau(x,z')r'})
(1, xz' = \tau(x,z'), -1, r') \ )
\end{align*}
where $x \in A \cup B \cup C$, $(l,r) \in Q$, and $l \equiv zl'$, $r \equiv z'r'$ where $z,z' \in Z$. Let $X_1'$ be the set of parallel paths of $\Gamma_S$ of the form:
\begin{align*}
( \ (1, l=r, +1, x),
(l', zx = \sigma(z,x), +1, 1)
\quad \quad \quad \quad \quad \quad \quad \quad \quad \quad \quad \quad \quad \quad
\\
\quad \quad \quad \quad \quad \quad \quad \quad \quad \quad \quad \quad \quad \quad
(\bbp_{l' \sigma(z,x), r'\sigma(z',x)})
(r', z'x = \sigma(z',x), -1, 1) \ )
\end{align*}
where $x \in A \cup B \cup C$, $(l,r) \in Q$, and $l \equiv l'z$, $r \equiv r'z'$ where $z,z' \in Z$. Note that $X_1$ and $X_1'$ are both finite when $Q$, $A$, $B$, and $C$ are all finite.

\begin{lem}\label{Last}
Let $\bbp$ be a non-empty path in $\Gamma_T$ and let $u,v \in (A \cup B \cup C \cup Z)^*$ such that $uv$ is non-empty. Then there exist paths $\bbq_1, \bbq_2 \in P_+(\Gamma_0)$, words $u', v' \in (A \cup B \cup C \cup Z)^*$, and a path $\bbp'$ in $\Gamma_T$ such that
\[
u \cdot \bbp \cdot v \sim_{X_1 \cup X_1'} \bbq_1 (u' \cdot \bbp' \cdot v') \bbq_2^{-1}
\]
and with $|u'v'| < |uv|$.
\end{lem}
\begin{proof}
Let $\bbp = \bbe_1 \circ \bbe_2 \circ \ldots \circ \bbe_n$ be a non-empty path in $\Gamma_T$ and let $u,v \in (A \cup B \cup C \cup Z)^*$ such that $uv$ is non-empty. Suppose that $u$ is non-empty. The case that $v$ is non-empty is dealt with using a dual argument and the set of parallel paths $X_1'$. Decompose $u \equiv u'x$ where $x \in A \cup B \cup C \cup Z$ and $u'$ is a, possibly empty, word over the same alphabet.

If $x \in Z$ then the path $(x \cdot \bbp)$ also belongs to $\Gamma_T$ and the result follows by setting $\bbq_1, \bbq_2$ to be empty, $\bbp' := (x \cdot \bbp)$, $u' = u'$ and $v' = v$.

Now suppose $x \not\in Z$. Let $z_i$ be the first letter of $\iota \bbe_i$ for $i=1,2 \ldots, n$, and let $z_{n+1}$ be the first letter of $\tau \bbe_n$, noting that each $z_i$ belongs to $Z$. For each $i = 1, \ldots, n+1$ define
\[
\bbf_i = (u', x z_i = \tau(x,z_i), +1, w_i)
\]
where the $w_i$ are obtained by the identities $\iota \bbe_i \equiv z_i w_i$ (for $i= 1, \ldots, n$) and $\tau \bbe_n \equiv z_{n+1} w_{n+1}$. Note that we can have $xz_i \neq xz_{i+1}$ in $S$ for some (or even all) values of $i$.

Now  for each $i \in \{1, \ldots, n\}$ we shall define a path $\bbs_i$ of $\Gamma_T$. For each $i \in \{1, \ldots, n\}$ let $\bbe_i = (\gamma_i, l_i = r_i, \epsilon_i, \delta_i)$. The definition of $\bbs_i$ varies depending on whether or not $\gamma_i$ is empty.

Suppose first that $\gamma_i$ is non-empty, so that $\gamma_i \equiv z_i \gamma_i'$ and $z_{i+1} \equiv z_i$ which implies $\tau(x,z_i) \equiv \tau(x,z_{i+1})$. In this case we define
\[
\bbs_i = (\tau(x,z_i) \gamma_i', l_i = r_i, \epsilon_i, \delta_i  )
\]
which is an edge of $\Gamma_T$. Applying (H1) we observe that
\[
(u \cdot \bbe_i \cdot v) \bbf_{i+1} \sim \bbf_i (u'\tau(x,z_i) \gamma_i', l_i = r_i, \epsilon_i, \delta_i v ).
\]
Rearranging this expression gives
\[
u \cdot \bbe_i \cdot v \sim \bbf_i (u' \cdot \bbs_i \cdot v) \bbf_{i+1}^{-1}.
\]
On the other hand, if $\gamma_i$ is empty, so that $l_i \equiv z_i l_i'$ and $r_i \equiv z_{i+1} r_i'$, we define
\[
\bbs_i = \bbp_{\tau(x,z_i)l_i', \tau(x,z_{i+1})r_i'} \cdot \delta_i,
\]
which is a path in $\Gamma_T$. By the definition of $X_1$ we have 
$$
\begin{array}{rcl}
u \cdot \bbe_i \cdot v & = & (u'x \gamma_i, l_i = r_i, \epsilon_i, \delta_i v) \\
& = &  u' \cdot (x, l_i = r_i, \epsilon_i, 1) \cdot \delta_i v \\
& \sim_{X_1} & u' \cdot (1, x z_i = \tau(x, z_i), +1, l_i') \cdot \delta_i v  \\
& & \quad \quad  \quad  u' \cdot \bbp_{\tau(x, z_i)l_i', \tau(x,z_{i+1}) r_i'} \cdot \delta_i v  \\
& & \quad \quad  \quad  \quad \quad  \quad  u' \cdot (1, x z_{i+1} = \tau(x,z_{i+1}) , -1, r_i') \cdot \delta_i v \\
& = & \bbf_i (u' \cdot \bbs_i \cdot v) \bbf_{i+1}^{-1}.
\end{array}
$$
Combining these observations completes the proof of the lemma since 
$$
\begin{array}{rcl}
u \cdot \bbp \cdot v & \sim &
(u \cdot \bbe_1 \cdot v) \ (u \cdot \bbe_2 \cdot v) \ \ldots \ (u \cdot \bbe_{n} \cdot v) \\
& \sim_{X_1} & 	
		\bbf_1 (u' \cdot \bbs_1 \cdot v) \bbf_2^{-1} 
        \bbf_2 (u' \cdot \bbs_2 \cdot v) \bbf_3^{-1} \bbf_3^{-1} \ldots
		\bbf_n^{-1} \bbf_n (u' \cdot \bbs_n \cdot v) \bbf_{n+1}^{-1} \\
& \sim & \bbf_1  (u' \cdot \bbs_1 \bbs_2 \ldots \bbs_n \cdot v)  \bbf_{n+1}^{-1} \\
& = & \bbq_1  (u' \cdot \bbp' \cdot v')  \bbq_2
\end{array}
$$
where $\bbq_1 = \bbf_1$, $\bbq_2 = \bbf_{n+1}^{-1}$, $v' \equiv v$ and $\bbp' = \bbs_1 \bbs_2 \ldots \bbs_n$ is a path in $\Gamma_T$.\end{proof}

\begin{cor}\label{LastCorol}
Let $\bbp$ be a path in $\Gamma_T$ and let $u,v \in (A \cup B \cup C \cup Z)^*$. Then there exist paths $\bbp_1$ and $\bbp_2$ in $P_+(\Gamma_0)$ and a path $\bbq$ in $P(\Gamma_T)$ such that 
\[
u \cdot \bbp \cdot v \sim_{X_1 \cup X_1'} \bbp_1 \bbq \bbp_2^{-1}.
\]
\end{cor} 
\begin{proof}
If $\bbp$ is empty then the result holds trivially by applying Lemma~\ref{Lemma1}(ii). Otherwise the result follows by repeated application of Lemma~\ref{Last}. 
\end{proof}

\noindent \textbf{The parallel paths $\mathbf X_2$.} For every edge $\bbe$ in $\Gamma_S$ by Lemmas~\ref{Lemma1} and \ref{dagger} we can fix paths $\bbp_{\iota \bbe}$, $\bbp_{\tau \bbe}$ in $P_+(\Gamma)$ and $\bbq_{\bbe}$ such that 
\[
\bbq_{\bbe} \in
\begin{cases}
\overline{B^1 \cdot P(\Gamma_G) \cdot C^1} & \mbox{if $\iota \bbe$ represents an element of $S \setminus T$} \\
P(\Gamma_T) & \mbox{if $\iota \bbe$ represents an element of $T$}
\end{cases}
\]
and where $\iota \bbp_{\iota \bbe} \equiv \iota \bbe$, $\iota \bbp_{\tau \bbe} \equiv \tau \bbe$, $\tau \bbp_{\iota \bbe} \equiv \iota \bbq_{\bbe}$ and $\tau \bbp_{\tau \bbe} \equiv \tau \bbq_{\bbe}$. Let $X_2$ be the set of parallel paths 
\[
\{ (\bbe, \bbp_{\iota \bbe} \bbq_{\bbe} \bbp_{\tau \bbe}^{-1}):
\bbe = (\alpha, r, \epsilon, \beta) \in P(\Gamma_S), \alpha, \beta \in A \cup B \cup C \cup Z \cup \{ 1 \} \ \& \ \epsilon \in \{ +1, -1 \} \}.
\]
Note that $X_2$ is finite when $\mathcal{P}_S$ is a finite presentation. 
\begin{lem}\label{Lemma3}
Let $\bbe$ be an edge in $\Gamma_S$.
\begin{enumerate}[(i)]
\item If $\iota \bbe$ represents an element of $S \setminus T$ then there exist paths $\bbp_3, \bbp_4 \in P_+(\Gamma)$ and $\bbq \in \overline{B^1 \cdot P(\Gamma_G) \cdot C^1}$ such that
$
\bbe \sim_{X} \bbp_3 \bbq \bbp_4^{-1}.
$
\item If $\iota \bbe$ represents an element of $T$ then there exist paths $\bbp_3, \bbp_4 \in P_+(\Gamma)$ and $\bbq \in P(\Gamma_T)$ such that
$
\bbe \sim_{X} \bbp_3 \bbq \bbp_4^{-1}.
$
\end{enumerate}
\end{lem}
\begin{proof}
(i) Let $\bbe = (u,r,\epsilon,v)$ be an arbitrary edge of the graph $\Gamma_S$. Since $\iota \bbe$ represents an element of $S \setminus T$, and since $T$ is an ideal, it follows that every subword of the word $\iota \bbe$ also represents an element of $S \setminus T$. It now follows by Lemma~\ref{Lemma1}(i) there is a path $\bbp_1 \in P_+(\Gamma)$ from $u$ to some $u' \in B^1 A^* C^1$, and a path $\bbp_2 \in P_+(\Gamma)$ from $v$ to some $v' \in B^1 A^* C^1$. Appling (H1) gives
\[
\bbe \sim
(\bbp_1 \cdot r_{\epsilon} v) (u' r_{\epsilon} \cdot \bbp_2)
(u', r, \epsilon, v')
(\bbp_1^{-1} \cdot r_{-\epsilon} v') (u r_{-\epsilon} \cdot \bbp_2^{-1}).
\]
Now consider the edge $\bbe' = (u', r, \epsilon, v')$. Write $u' \equiv u'' \alpha$ and $v' \equiv \beta v''$ where $\alpha, \beta \in A \cup B \cup C \cup \{1 \}$, and where $\alpha$ (resp. $\beta$) is empty only if $u'$ (resp. $v'$) is empty. Putting $\bbf = (\alpha, r, \epsilon, \beta)$, from the definition of $X_2$ we have
\[
\bbe' = u'' \cdot \bbf \cdot v'' \sim_{X_2} u'' \cdot \bbp_{\iota \bbf} \bbq_{\bbf} \bbp_{\tau \bbf}^{-1} \cdot v''.
\] \begin{sloppypar}
Next by applying Lemma~\ref{trick}, since $\iota \bbq_\bbf, \tau \bbq_\bbf \in B^1 \cdot A^* \cdot C^1$ we deduce that $u'' \iota \bbq_{\bbf} v'', u'' \tau \bbq_{\bbf} v'' \in B^1 A^* C^1$ and by the definitions of $\bbq_{\bbf}$ and of $\overline{B^1 \cdot P(\Gamma_G) \cdot C^1}$ this implies  
$$
u'' \cdot \bbq_\bbf \cdot v'' \in \overline{B^1 \cdot P(\Gamma_G) \cdot C^1}.
$$
This completes the proof of part (i) since we have  
$
\bbe \sim_{X_2} \bbp_3 \bbq \bbp_4^{-1}
$
where $\bbp_3 = (\bbp_1 \cdot r_{\epsilon} v) (u' r_{\epsilon} \cdot \bbp_2) (u'' \cdot \bbp_{\iota \bbf} \cdot v'')$,
$\bbp_4 = (u r_{-\epsilon} \cdot \bbp_2) (\bbp_1 \cdot r_{-\epsilon} v') (u'' \cdot \bbp_{\tau \bbf} \cdot v'')$, and  
$\bbq = u'' \cdot \bbq_\bbf \cdot v'' \in \overline{B^1 \cdot P(\Gamma_G) \cdot C^1}$. \end{sloppypar}

(ii) We proceed in much the same way as in part (i) except that here we need also to apply Corollary~\ref{LastCorol}.  Let $\bbe = (u,r,\epsilon,v)$ be an arbitrary edge of the graph $\Gamma_S$. It now follows by Lemma~\ref{Lemma1} that there is a path $\bbp_1 \in P_+(\Gamma)$ from $u$ to some $u' \in B^1 A^* C^1 \cup Z^+$ (depending on whether $u$ represents an element of $T$ or of $S \setminus T$), and a path $\bbp_2 \in P_+(\Gamma)$ from $v$ to some $v' \in B^1 A^* C^1 \cup Z^+$. Applying (H1) gives
\[
\bbe \sim
(\bbp_1 \cdot r_{\epsilon} v) (u' r_{\epsilon} \cdot \bbp_2)
(u', r, \epsilon, v')
(\bbp_1^{-1} \cdot r_{-\epsilon} v') (u r_{-\epsilon} \cdot \bbp_2^{-1}).
\]
Now consider the edge $\bbe' = (u', r, \epsilon, v')$. Write $u' \equiv u'' \alpha$ and $v' \equiv \beta v''$ where $\alpha, \beta \in Z \cup A \cup B \cup C \cup \{1 \}$, and where $\alpha$ (resp. $\beta$) is empty only if $u'$ (resp. $v'$) is empty. It now follows from Lemma~\ref{trick} that the word $\alpha r_{\epsilon} \beta$ (and hence also the word $\alpha r_{-\epsilon} \beta$) represents an element of $T$. Indeed, if $\alpha r_{\epsilon} \beta$ represented an element of $S \setminus T$ then by Lemma~\ref{trick} it would follow that $u r_{\epsilon} v$ represents an element of $S \setminus T$, contrary to the assumption that $\iota \bbe$ represents an element of $T$.  

Now putting $\bbf = (\alpha, r, \epsilon, \beta)$ from definition of $X_2$ we have
\[
\bbe' \sim_{X_2} u'' \cdot \bbp_{\iota \bbf} \bbq_{\bbf} \bbp_{\tau \bbf}^{-1} \cdot v''
\]
Since both $\alpha r_{\epsilon} \beta$ and $\alpha r_{-\epsilon} \beta$ represent an element of $T$ it follows that $\iota \bbq_{\bbf}, \tau \bbq_{\bbf} \in Z^+$ and hence by definition $\bbq_{\bbf} \in P(\Gamma_T)$. Now by Corollary~\ref{LastCorol} applied to $u'' \cdot \bbq_{\bbf} \cdot v''$ there are paths $\bbp_5, \bbp_6 \in P_+(\Gamma)$ and $\bbq \in P(\Gamma_T)$ such that
\[
u'' \cdot \bbq_\bbf \cdot v'' \sim_X \bbp_5 \bbq \bbp_6^{-1}.
\]
This completes the proof of part (ii) since we have 
$
\bbe \sim_{X} \bbp_3 \bbq \bbp_4^{-1}
$
where $\bbp_3 = (\bbp_1 \cdot r_{\epsilon} v) (u' r_{\epsilon} \cdot \bbp_2) (u'' \cdot \bbp_{\iota \bbf} \cdot v'') \bbp_5$,
$\bbp_4 = (u r_{-\epsilon} \cdot \bbp_2) (\bbp_1 \cdot r_{-\epsilon} v') (u'' \cdot \bbp_{\tau \bbf} \cdot v'') \bbp_6$, and $\bbq \in P(\Gamma_T)$.
\end{proof}

\noindent \textbf{The parallel paths $\mathbf X_3$.} For any pair of positive edges $\bbe_1, \bbe_2$ from $P_+(\Gamma)$ with 
$\iota \bbe_1 \equiv \iota \bbe_2$, by Lemma~\ref{Lemma1} and Lemma~\ref{dagger} we can fix paths $\bbp_{\tau \bbe_1}, \bbp_{\tau \bbe_2} \in P_+(\Gamma)$ and 
\[
\bbq_{\bbe_1, \bbe_2} \in 
\begin{cases}
\overline{B^1 \cdot P(\Gamma_G) \cdot C^1} & \mbox{if $\iota \bbe_1$ represents an element of $S \setminus T$ } \\
P(\Gamma_T) & \mbox{if $\iota \bbe_1$ represents an element of $T$ }
\end{cases}
\]
such that $\tau \bbe_1 \equiv \iota \bbp_{\tau \bbe_1}$, $\tau \bbe_2 \equiv \iota \bbp_{\tau \bbe_2}$, $\tau \bbp_{\tau \bbe_1} \equiv \iota \bbq_{\bbe_1, \bbe_2}$ and $\tau \bbp_{\tau \bbe_1} \equiv \tau \bbq_{\bbe_1, \bbe_2}$. We then let $X_3$ be the following set of parallel paths:
\[
\{
(\bbe_1 \bbp_{\tau \bbe_1} \bbq_{\bbe_1, \bbe_2}, \bbe_2 \bbp_{\tau \bbe_2}   ): 
\bbe_1, \bbe_2 \in P_+(\Gamma), 
\iota \bbe_1 \equiv \iota \bbe_2 \ \& \ | \iota \bbe_1 | \leq 2 |e| + 3
\}.
\]
Recall that $e$ is the word from $A^+$ representing the identity of the group $G$ that appears in the presentation $\mathcal{P}_S$. The restriction $|\iota \bbe_1| \leq 2|e| + 3$ comes from consideration of lengths of words appearing on left hand sides of the relations in $R_0 \cup R_U \cup R_T$. 

If $\mathcal{P}_S$ is a finite presentation then $X_3$ is finite. To see this, first observe that for every vertex $w$ of $\Gamma_S$ the number of edges $\bbe$ of $\Gamma_S$ with $\iota \bbe \equiv w$ is finite, as $w$ has finite length and there are only finitely many relations that can be applied to $w$. So since in the definition of $X_3$, $\iota \bbe_1$ and $\iota \bbe_2$ can be only one of finitely many words, since the generating set in the presentation $\mathcal{P}_S$ is finite, it follows that there are only finitely many possibilities for $\bbe_1$ and $\bbe_2$. Hence $X_3$ is finite. 

For the following proof the notion of the applications of relations to a word ``overlapping'' is important. For two edges $\bbe_1, \bbe_2 \in P_+(\Gamma)$ with $\iota \bbe_1 \equiv \iota \bbe_2$ we say that $\bbe_1$ and $\bbe_2$ do \emph{not overlap} if we can write 
$\iota \bbe_1 \equiv u r_{+1} v s_{+1} w$, where $u,v,w \in (A \cup B \cup C \cup Z)^*$ and 
\[
\{ \bbe_1, \bbe_2  \} = \{(u, r_{+1} = r_{-1}, +1, v s_{+1} w), (u r_{+1} v, s_{+1} = s_{-1}, +1, w)  \}. 
\]
Otherwise we say that $\bbe_1$ and $\bbe_2$ overlap.  

\begin{lem}\label{Lemma4}
Let $\bbe_1, \bbe_2 \in P_+(\Gamma)$ with $\iota \bbe_1 \equiv \iota \bbe_2$.
\begin{enumerate}[(i)]
\item If $\iota \bbe_1$ represents an element of $S \setminus T$ then there are paths $\bbp,\bbp' \in P_+(\Gamma)$ and $\bbq \in \overline{B^1 \cdot P(\Gamma_G) \cdot C^1}$ such that $\bbe_1 \bbp \bbq \sim_X \bbe_2 \bbp'$.
\item If $\iota \bbe_1$ represents an element of $T$ then there are paths $\bbp,\bbp' \in P_+(\Gamma)$ and $\bbq \in P(\Gamma_T)$ such that $\bbe_1 \bbp \bbq \sim_X \bbe_2 \bbp'$.
\end{enumerate}
\end{lem}
\begin{proof}
(i) If $\bbe_1 = \bbe_2$ then the result holds trivially by Lemma~\ref{Lemma1}, so suppose that $\bbe_1 \neq \bbe_2$.
\begin{sloppypar}
Next suppose that $\bbe_1$ and $\bbe_2$ do not overlap. So without loss of generality we have $\iota \bbe_1 \equiv \iota \bbe_2 \equiv u r_{+1} v s_{+1} w$, $\bbe_1 = (u, r_{+1} = r_{-1}, +1, vs_{+1} w)$, and $\bbe_2 = (ur_{+1} v, s_{+1} = s_{-1}, +1, w)$. Then define $\bbf_1 = (ur_{-1} v, s_{+1} = s_{-1}, +1, w)$ and $\bbf_2 = (u, r_{+1} = r_{-1}, +1, vs_{-1} w)$. By Lemma~\ref{Lemma1} we can fix a path $\bbp_1 \in P^+(\Gamma)$ from $u r_{-1} v s_{-1} w$ to a word $x \in B^1 A^* C^1$. Now set $\bbq$ to be the empty path $1_x$, and set $\bbp = \bbf_1 \bbp_1$ and $\bbp' = \bbf_2 \bbp_1$. Applying (H1) we have $\bbe_1 \bbf_1 \sim \bbe_2 \bbf_2$ and hence $\bbe_1 \bbp \bbq \sim \bbe_2 \bbp'$, as  required.  
\end{sloppypar}
Now suppose that $\bbe_1$ and $\bbe_2$ do overlap. It follows that there exist $\bbf_3, \bbf_4 \in P_+(\Gamma)$ and words $w_1, w_2 \in (A \cup B \cup C \cup Z)^*$ such that $\iota \bbf_3 = \iota \bbf_4$, $\bbe_1 = w_1 \cdot \bbf_3 \cdot w_2$, $\bbe_2 = w_1 \cdot \bbf_4 \cdot w_2$ and $|\iota \bbf_3 | = |\iota \bbf_4| \leq 2 |e| + 1$. This inequality comes from the fact that $\bbe_1$ and $\bbe_2$ overlap, and that the lengths of the words appearing as left hand sides of the relations in $R_0 \cup R_U \cup R_T$ are bounded above by $|e| + 1$. 

As in the proof of Lemma~\ref{Lemma3}(i), applying Lemma~\ref{Lemma1}(i) there is a path $\bbp_1 \in P_+(\Gamma)$ from $w_1$ to some word $w_1' \in B^1 A^* C^1$, and a path $\bbp_2 \in P_+(\Gamma)$ from $w_2$ to some $w_2' \in B^1 A^* C^1$. Write $w_1' \equiv w_1'' \alpha$ and $w_2' \equiv \beta w_2''$ where $\alpha, \beta \in A \cup B \cup C \cup \{1 \}$, and where $\alpha$ (resp. $\beta$) is empty only if $w_1'$ (resp. $w_2'$) is empty. Let
\[
\bbe_1' = w_1'' \cdot (\alpha \cdot \bbf_3 \cdot \beta) \cdot w_2'', \quad
\bbe_2' = w_1'' \cdot (\alpha \cdot \bbf_4 \cdot \beta) \cdot w_2'' 
\]
Applying (H1) we obtain
\begin{align}
\bbe_1 (\bbp_1 \cdot \tau \bbf_3 w_2)
(w_1' \tau \bbf_3 \cdot \bbp_2) 
 \sim 
(\bbp_1 \cdot \iota \bbf_3 w_2) 
(w_1' \iota \bbf_3 \cdot \bbp_2) \bbe_1' \label{first}
\end{align}
and
\begin{align}
\bbe_2 (\bbp_1 \cdot \tau \bbf_4 w_2)
(w_1' \tau \bbf_4 \cdot \bbp_2) 
 \sim 
(\bbp_1 \cdot \iota \bbf_4 w_2) 
(w_1' \iota \bbf_4 \cdot \bbp_2) \bbe_2' \label{second} \\
 =  (\bbp_1 \cdot \iota \bbf_3 w_2) 
(w_1' \iota \bbf_3 \cdot \bbp_2) \bbe_2' \notag
\end{align}
since $\iota \bbf_3 \equiv \iota \bbf_4$. Then by definition of $X_3$, and since $|\iota(\alpha \cdot \bbf_3 \cdot \beta)| = |\iota \bbf_3| + 2 \leq 2|e| + 3$ we have
\[
(
(\alpha \cdot \bbf_3 \cdot \beta)
\bbp_{\tau(\alpha \cdot \bbf_3 \cdot \beta)} \bbq_{\alpha \cdot \bbf_3 \cdot \beta, \alpha \cdot \bbf_4 \cdot \beta}, \
(\alpha \cdot \bbf_4 \cdot \beta)
\bbp_{\tau(\alpha \cdot \bbf_4 \cdot \beta)}
) \in X_3.
\]
Hence by acting on the left by $w_1''$ and on the right by $w_2''$ we obtain
\begin{align}
 \notag \bbe_1' (w_1'' \cdot \bbp_{\tau(\alpha \cdot \bbf_3 \cdot \beta)}
\bbq_{\alpha \cdot \bbf_3 \cdot \beta, \alpha \cdot \bbf_4 \cdot \beta} \cdot w_2'') \\ \sim_{X_3} &
\bbe_2' (w_1'' \cdot \bbp_{\tau(\alpha \cdot \bbf_4 \cdot \beta)} \cdot w_2''). \label{third}
\end{align}
To complete the proof of part (i) we now set 
\[
\bbp = (\bbp_1 \cdot \tau \bbf_3 w_2)
(w_1' \tau \bbf_3 \cdot \bbp_2)
(w_1'' \cdot  \bbp_{\tau(\alpha \cdot \bbf_3 \cdot \beta)}  \cdot w_2''),
\]
which is a path in $P_+(\Gamma)$, define
\[
\bbp' = (\bbp_1 \cdot \tau \bbf_4 w_2)
(w_1' \tau \bbf_4 \cdot \bbp_2)
(w_1'' \cdot  \bbp_{\tau(\alpha \cdot \bbf_4 \cdot \beta)}  \cdot w_2''),
\]
which is a path in $P_+(\Gamma)$, and set $\bbq = w_1'' \cdot \bbq_{\alpha \cdot \bbf_3 \cdot \beta, \alpha \cdot \bbf_4 \cdot \beta} \cdot w_2''$. As in the proof of Lemma~\ref{Lemma3}, to see that 
$\bbq \in \overline{B^1 \cdot P(\Gamma_G) \cdot C^1}$ we apply Lemma~\ref{trick} to $w_1'' (\alpha \tau \bbf_3 \beta) w_2''$ and $w_1'' (\alpha \tau \bbf_4 \beta) w_2''$, deducing that $\iota \bbq, \tau \bbq \in B^1 A^* C^1$, which along with the definition of $\bbq_{\alpha \cdot \bbf_3 \cdot \beta, \alpha \cdot \bbf_4 \cdot \beta}$ implies that $\bbq \in \overline{B^1 \cdot P(\Gamma_G) \cdot C^1}$.  

(ii) We proceed in much the same way as in part (i) except that here we need also to apply Corollary~\ref{LastCorol}. The cases when $\bbe_1 = \bbe_2$, or when $\bbe_1$ and $\bbe_2$ do not overlap, are dealt with in exactly the same way as in part (i).

Now suppose that $\bbe_1$ and $\bbe_2$ do overlap. As in part (i) it follows that there exist $\bbf_3, \bbf_4 \in P_+(\Gamma)$ and words $w_1, w_2 \in (A \cup B \cup C \cup Z)^*$ such that $\iota \bbf_3 \equiv \iota \bbf_4$, $\bbe_1 \equiv w_1 \cdot \bbf_3 \cdot w_2$, $\bbe_2 = w_1 \cdot \bbf_4 \cdot w_2$ and $|\iota \bbf_3 | = |\iota \bbf_4| \leq 2 |e| + 1$. 

As in the proof of Lemma~\ref{Lemma3}(ii), applying Lemma~\ref{Lemma1} there is a path $\bbp_1 \in P_+(\Gamma)$ from $w_1$ to some word $w_1' \in B^1 A^* C^1 \cup Z^+$, and a path $\bbp_2 \in P_+(\Gamma)$ from $w_2$ to some $w_2' \in B^1 A^* C^1 \cup Z^+$. Write $w_1' \equiv w_1'' \alpha$ and $w_2' \equiv \beta w_2''$ where $\alpha, \beta \in Z \cup A \cup B \cup C \cup \{1 \}$, and where $\alpha$ (resp. $\beta$) is empty only if $w_1'$ (resp. $w_2'$) is empty. Let $\bbe_1'$ and $\bbe_2'$ be defined in the same way as in part (i). Then in exactly the same way as in part (i) we can prove that the equations \eqref{first}, \eqref{second} and \eqref{third} all hold. Here we set $\bbq' = w_1'' \cdot \bbq_{\alpha \cdot \bbf_3 \cdot \beta, \alpha \cdot \bbf_4 \cdot \beta} \cdot w_2''$. 

Next, in contrast to part (i), we claim that here $\bbq_{\alpha \cdot \bbf_3 \cdot \beta, \alpha \cdot \bbf_4 \cdot \beta} \in P(\Gamma_T)$. To see this it is sufficient to show that $\alpha \tau \bbf_3 \beta$ and $\alpha \tau \bbf_4 \beta$ both represent elements of $T$. If $\alpha$ or $\beta$ belongs to $Z$ then this is obvious. Otherwise, by Lemma~\ref{trick} if $\alpha \tau \bbf_3 \beta$ represented an element of $S \setminus T$ then it would follow that $\iota \bbq'$ also represents an element of $S \setminus T$, a contradiction. Similarly we see that $\alpha \tau \bbf_4 \beta$ must represent an element of $T$, and this completes the proof that $\bbq_{\alpha \cdot \bbf_3 \cdot \beta, \alpha \cdot \bbf_4 \cdot \beta} \in P(\Gamma_T)$. 

Finally, by Corollary~\ref{LastCorol} there exist paths $\bbp_3, \bbp_4 \in P_+(\Gamma)$ and $\bbq \in P(\Gamma_T)$ such that $\bbq' \sim_{X} \bbp_3 \bbq \bbp_4^{-1}$. This completes the proof of the lemma, with $\bbq$ defined as above, and by setting
\[
\bbp = (\bbp_1 \cdot \tau \bbf_3 w_2)
(w_1' \tau \bbf_3 \cdot \bbp_2)
(w_1'' \cdot  \bbp_{\tau(\alpha \cdot \bbf_3 \cdot \beta)}  \cdot w_2'')
\bbp_3
\] 
and
\[
\bbp' = (\bbp_1 \cdot \tau \bbf_4 w_2)
(w_1' \tau \bbf_4 \cdot \bbp_2)
(w_1'' \cdot  \bbp_{\tau(\alpha \cdot \bbf_4 \cdot \beta)}  \cdot w_2'')
\bbp_4,
\]
which are both paths in $P_+(\Gamma)$. 
\end{proof}
The next thing we want to do is to extend Lemma~\ref{Lemma4} to a statement about paths. To do this we need to introduce a function on words that, roughly speaking, gives a measure of how close that word is to being in quasi-normal form $B^1 A^* C^1 \cup Z^+$. 

As in the proof of Lemma~\ref{Lemma1}, given a word $w \in W^+$ and a subset $V \subseteq W$ we write $|w|_V$ to denote the total number of letters of $w$ that belong to $V$. Now we define  
$$
\begin{array}{rcl}
F: (A \cup B \cup C \cup Z)^+ & \rightarrow & \mathbb{N} \times \mathbb{N} \\
F(w) & = & (|w|_{B \cup C}, |w|_A).
\end{array}
$$
Note that $F(w)$ does not depend in any way on the number of letters from $Z$ in $w$. 
Now define the following total order $<$ on $\mathbb{N} \times \mathbb{N}$ by
\[
(n_1,m_1) < (n_2,m_2) \Leftrightarrow
\begin{cases}
n_1 < n_2 \ \mbox{or} \\
n_1 = n_2 \ \mbox{and} \ m_1 < m_2.
\end{cases}
\]
Clearly with this ordering $(\mathbb{N} \times \mathbb{N}, <)$ is well-founded (since it is the lexicographic product of two well-founded orders) with a unique minimal element $(0,0)$, which allows us to apply \emph{well-founded induction} (as we shall do below in the proof of Lemma~\ref{Lemma5}). 

Also define addition on pairs of natural numbers in the usual way where
\[
(n_1,m_1) + (n_2,m_2) = (n_1+n_2, m_1 + m_2).
\]
The idea behind the definition of $F$ is that given a vertex of $\Gamma$ if there are no edges from $P_+(\Gamma)$ coming out of it, it means that that the vertex belongs either to $B^1 A^* C^1 \cup Z^+$, and as the number $F(w)$ decreases, the closer we are to having our word in quasi-normal form.
\begin{lem}\label{FProp}
The function $F$ has the following properties.
\begin{enumerate}[(i)]
\item $F(w_1w_2) = F(w_1) + F(w_2)$ for all words $w_1, w_2$;
\item $F(w) = (0,0)$ if and only if $w \in Z^+$;
\item  $F(\tau \bbe) < F(\iota \bbe)$ for all $\bbe \in P_+(\Gamma)$.
\end{enumerate}
\end{lem}
\begin{proof}
Parts (i) and (ii) follow straight from the definitions. Part (iii) is proved by using part (i) and checking that (iii) holds for the edges $\bbe = (1,r,+1,1)$ where $r \in R_0 \cup R_U \cup R_T$. 
\end{proof}
\begin{lem}\label{Lemma5}
Let $\bbp_1, \bbp_2 \in P_+(\Gamma)$ with $\iota \bbp_1 \equiv \iota \bbp_2$.
\begin{enumerate}[(i)] 
\item If $\iota \bbp_1$ represents an element of $S \setminus T$ then there are paths $\bbp,\bbp' \in P_+(\Gamma)$ and $\bbq \in \overline{B^1 \cdot P(\Gamma_G) \cdot C^1}$ such that $\bbp_1 \bbp \bbq \sim_X \bbp_2 \bbp'$.
\item If $\iota \bbp_1$ represents an element of $T$ then there are paths $\bbp,\bbp' \in P_+(\Gamma)$ and $\bbq \in P(\Gamma_T)$ such that $\bbp_1 \bbp \bbq \sim_X \bbe_2 \bbp'$.
\end{enumerate}
\end{lem}
\begin{proof}
For $m,n \in \mathbb{N}$, let $R(m,n)$ denote the following statement:
\begin{quote}
$R(m,n)$: Lemma~\ref{Lemma5} holds for all paths $\bbp_1$ and $\bbp_2$ such that $F(\iota \bbp_1) = (m,n)$. 
\end{quote}
We shall prove $R(m,n)$ for all $n,m \in \mathbb{N}$ by well-founded induction. First note that $R(0,0)$ holds trivially. Now we shall prove 
\begin{quote}
For any $m,n \in \mathbb{N}$, if $R(x,y)$ holds for all $(x,y) < (m,n)$ then $R(m,n)$ holds. 
\end{quote} 
First consider part (i). Let $\bbp_1, \bbp_2 \in P_+(\Gamma)$ with $\iota \bbp_1 \equiv \iota \bbp_2$ and where $\iota \bbp_1$ represents an element of $S \setminus T$. If $\iota \bbp_1$ belongs to $B^1 A^* C^1$ then the result holds trivially, so suppose otherwise. Also, if one of $\bbp_1$ or $\bbp_2$ is the empty path then the result is easily seen to be true by applying Lemma~\ref{Lemma1}, so suppose not. So we can write $\bbp_1 = \bbe_1 \bbp_1'$ and $\bbp_2 = \bbe_2 \bbp_2'$ where $\bbe_i$ are edges in $P_+(\Gamma)$ and $\bbp_1'$ and  $\bbp_2'$ both belong to $P_+(\Gamma)$. By Lemma~\ref{Lemma4}(i) there are  paths $\bbp_3, \bbp_4 \in \mathcal{P}_+(\Gamma)$ and $\bbq' \in \overline{B^1 \cdot P(\Gamma_G) \cdot C^1}$ such that 
$
\bbe_1 \bbp_3 \bbq' \sim_X \bbe_2 \bbp_4.
$
Now $\bbp_1', \bbp_3 \in P_+(\Gamma)$, $\iota \bbp_1' \equiv \iota \bbp_3$ and by Lemma~\ref{FProp} we have $F(\iota \bbp_1') < F(\iota \bbp_1)$. So by the inductive hypothesis there exist paths $\bbp_1'' \in P_+(\Gamma)$ and $\bbq_1 \in \overline{B^1 \cdot P(\Gamma_G) \cdot C^1}$ such that
$
\bbp_1' \bbp_1'' \bbq_1 \sim_X \bbp_3.
$
Likewise there exist paths $\bbp_2'' \in P_+(\Gamma)$ and $\bbq_2 \in \overline{B^1 \cdot P(\Gamma_G) \cdot C^1}$ such that
$
\bbp_4 \bbq_2 \sim \bbp_2' \bbp_2''.
$
Combining these observations together, we have found $\bbp_1'', \bbp_2'' \in P_+(\Gamma)$, and $\bbq_1 \bbq' \bbq_2 \in \overline{B^1 \cdot P(\Gamma_G) \cdot C^1}$ such that
\[
\bbp_1 \bbp_1'' \bbq_1 \bbq' \bbq_2 \sim_X \bbp_2 \bbp_2'',
\]
as required.

The proof of part (ii) is done in exactly the same way as the proof of part (i) but by applying Lemma~\ref{Lemma4}(ii).
\end{proof}

The following corollary is now immediate. 

\begin{cor}\label{Corollary}
Let $\bbp_1, \bbp_2 \in P_+(\Gamma)$ with $\iota \bbp_1 \equiv \iota \bbp_2$.
\begin{enumerate}[(i)]
\item If $\iota \bbp_1$ represents an element of $S \setminus T$ and $\tau \bbp_1 , \tau \bbp_2  \in B^1 A^* C^1$ then there is a path $\bbq \in \overline{B^1 \cdot P(\Gamma_G) \cdot C}$ such that $\bbp_1 \sim_X \bbp_2 \bbq$.
\item If $\iota \bbp_1$ represents an element of $T$ and $\tau \bbp_1 , \tau \bbp_2  \in Z^+$ then there is a path $\bbq \in P(\Gamma_T)$ such that $\bbp_1 \sim_X \bbp_2 \bbq$.
\end{enumerate}
\end{cor}

\begin{lem}\label{Lemma6}
Let $\bbp$ be any path in $\Gamma_S$.
\begin{enumerate}[(i)]
\item If $\iota \bbp$ represents an element of $S \setminus T $ 
then there exist paths $\bbp_1', \bbp_m \in P_+(\Gamma)$ 
and $\bbq \in \overline{B^1 \cdot P(\Gamma_G) \cdot C^1}$ such that $\bbp \sim_X \bbp_1' \bbq \bbp_m^{-1}$.
\item If $\iota \bbp$ represents an element of $T$ then there exist paths $\bbp_1', \bbp_m \in P_+(\Gamma)$ and  $\bbq \in P(\Gamma_T)$ such that $\bbp \sim_X \bbp_1' \bbq \bbp_m^{-1}$.
\end{enumerate}
\end{lem}
\begin{proof}
(i) If $\bbp$ is empty then the result holds trivially by Lemma~\ref{Lemma1}, so suppose otherwise. 
Let $\bbp = \bbe_1 \bbe_2 \ldots \bbe_m$, where each $\bbe_i$ is an edge of $\Gamma_S$. By Lemma~\ref{Lemma3}(i) there are paths $\bbp_i, \bbp_i' \in P_+(\Gamma)$ and $\bbq_i \in \overline{B^1 \cdot P(\Gamma_G) \cdot C^1}$, such that $\tau \bbp_i, \tau \bbp_i' \in B^1 A^* C^1$ and $\bbe_i \sim_X \bbp_i' \bbq_i \bbp_i^{-1}$ for $i=1,\ldots,m$. Hence
\[
\bbp \sim_X
\bbp_1' \bbq_1 \bbp_1^{-1}
\bbp_2' \bbq_2 \bbp_2^{-1}
\ldots
\bbp_m' \bbq_m \bbp_m^{-1}.
\]
Since $\iota \bbp_i \equiv \iota \bbp_{i+1}'$ and $\tau \bbp_i, \tau \bbp_{i+1}' \in B^1 A^* C^1$, by Corollary~\ref{Corollary}(i) there is a path $\bbq_i' \in \overline{B^1 \cdot \mathcal{P}(\Gamma_G) \cdot C^1}$ such that $\bbp_{i+1}' \sim_X \bbp_i \bbq_i'$ so that $\bbp_i^{-1} \bbp_{i+1}' \sim_X \bbq_i'$. Thus
\[
\bbp \sim_X
\bbp_1'
\bbq_1 \bbq_1'
\bbq_2 \bbq_2'
\ldots
\bbq_{m-1} \bbq_{m-1}'
\bbp_m^{-1},
\]
as required.

The proof of part (ii) is done in exactly the same way as part (i), except that we apply 
Lemma~\ref{Lemma3}(ii)
and
Corollary~\ref{Corollary}(ii). \end{proof}

\noindent \textbf{The parallel paths $\mathbf X_e$.} Let $X_e$ denote the following set of parallel paths of $\Gamma_S$ 
\[
\{ (\ (b_i, e c_{\lambda} = c_{\lambda}, +1, 1), \  (1, b_i e = b_i, +1, c_{\lambda}) \ ) : i \in I \setminus \{ 1 \}, \lambda \in \Lambda \setminus \{ 1 \} \}.
\]
Note that $X_e$ is finite when both $I$ and $\Lambda$ are finite.

\begin{lem}\label{newlemma}
Let $\bbq$ be a path in $\overline{B^1 \cdot P(\Gamma_G) \cdot C^1}$. 
\begin{enumerate}[(i)]
\item If neither $\iota \bbq$ nor $\tau \bbq$ belongs to $B \cup C \cup BC$ then there exists a path $\bbq'$ in $B^1 \cdot P(\Gamma_G) \cdot C^1$ such that $\bbq \sim_{X_e} \bbq'$.
\item If both $\iota \bbq$ and $\tau \bbq$ belong to $B \cup C \cup BC$ then it follows that $\iota \bbq \equiv \tau \bbq$, and there exists a path $\bbq'$ in $B^1 \cdot P(\Gamma_G) \cdot C^1$ and edges $\bbe_1$, $\bbe_2$ of $\overline{B^1 \cdot P(\Gamma_G) \cdot C^1}$ such that $\bbq \sim_{X_e} \bbe_1 \bbq' \bbe_2$ and $\bbe_1 \sim_{X_e} \bbe_2^{-1}$.  
\end{enumerate}
\end{lem}
\begin{proof}

(i) Suppose that neither $\iota \bbq$ nor $\tau \bbq$ belongs to $B \cup C \cup BC$. 
If $\bbq$ is empty or is already a path in $B^1 \cdot P(\Gamma_G) \cdot C^1$ then the result is trivial. Otherwise there exists an edge $\bbe$ of the path $\bbq$ with $\tau \bbe \in B \cup C \cup BC$. If $\tau \bbe \in B \cup C$ then the edge $\bbe$ must be of the form $(1, b_i e = b_i, +1, 1)$ or $(1, ec_{\lambda} = c_{\lambda}, +1, 1)$ for some $i \in I \setminus \{ 1 \}$, $\lambda \in \Lambda \setminus \{ 1 \}$. By consideration of the relations in the presentation $\mathcal{P}_S$ it follows that in the path $\bbq$ the edge $\bbe$ must be immediately followed by its inverse $\bbe^{-1}$. On the other hand, if $\tau \bbe \in BC$ then the edge $\bbe$ equals either $\bbe_c = (b_i, ec_{\lambda}, +1, 1)$ or $\bbe_b = (1, b_i e = b_i, +1, c_{\lambda})$, for some $i \in I \setminus \{ 1 \}$, $\lambda \in \Lambda \setminus \{ 1 \}$. In this case $\tau \bbe \equiv b_i c_{\lambda}$ and in the path $\bbq$ the edge $\bbe$ must be immediately followed by $\bbe_c^{-1}$ or $\bbe_b^{-1}$. Since $(\bbe_c, \bbe_b) \in X_e$ it follows that 
in any of these situations we can remove $\bbe$ and the edge that follows it from the path $\bbq$ to obtain a shorter path that is $\sim_{X_e}$ related to the original. Repeating this process, after a finite number of steps we obtain the desired path $\bbq'$ from $B^1 \cdot P(\Gamma_G) \cdot C^1$ with $\bbq \sim_{X_e} \bbq'$. 

(ii) Suppose that $\iota \bbq, \tau \bbq \in B \cup C \cup BC$. The fact that $\iota \bbq \equiv \tau \bbq$ follows from Lemma~\ref{dagger}. If $\bbq$ is empty then the result holds trivially by setting $\bbe_1$ to be any edge with initial vertex $\iota \bbq$ and terminal vertex in $BA^+$ (or $CA^+$ or $B A^+ C$ depending on the value of $\iota \bbq$) and letting $\bbe_2 = \bbe_1^{-1}$. Now suppose that $\bbq$ is not empty, and let $\bbe$ be its first edge. Now $\bbq$ is a closed path, and $\iota \bbq$ belongs either to $B$, $C$ or $BC$. In the first two cases there is only one edge in $\overline{B^1 \cdot P(\Gamma_G) \cdot C^1}$ with initial vertex $\iota \bbe_1$ and hence since $\bbq$ is closed, the last edge of $\bbq$ must be $\bbe_1^{-1}$. Then the result follows by applying part (i). 

On the other hand, if $\iota \bbe_1 \in BC$, then the last edge $\bbe_2$ of $\bbq$ satisfies $\tau \bbe_1 \equiv \iota \bbe_2 \equiv b_i e c_{\lambda}$, for some $i \in I$, $\lambda \in \Lambda$, and either $\bbe_1 = \bbe_2^{-1}$ or $(\bbe_1, \bbe_2) \in X_e$. Now the result follows by applying part (i).  \end{proof}

We are now in a position to complete the proof of Theorem~\ref{comp0simpExt}. With the definitions given above we define $X = X_1 \cup X_1' \cup X_2 \cup X_3 \cup X_e \cup X_G \cup X_T$ where $X_G$ is a homotopy base for the presentation $\lb A | R \rb$ of $G$, and $X_T$ is a homotopy base for the presentation $\lb Z | Q \rb$ of $T$. Here $X_G$ is a set of parallel paths in the subgraph $\Gamma_G$ of $\Gamma_S$, and $X_T$ is a set of parallel paths in the subgraph $\Gamma_T$ of $\Gamma_S$.    

The proof of Theorem~\ref{comp0simpExt} is then wrapped up with the following lemma.

\begin{lem}\label{wrappingup}
The set $X$ is a homotopy base for the presentation $\mathcal{P}_S$ of $S$.  
\end{lem}
\begin{proof}
Let $\bbp_1, \bbp_2 \in P(\Gamma_S)$ with $(\bbp_1,\bbp_2) \in \parallel$. We must prove that $\bbp_1 \sim_X \bbp_2$. There are two main cases depending on whether or not $\iota \bbp_1$ represents an element of $T$.

First suppose that  $\iota \bbp_1$ represents an element of $S \setminus T$. By Lemma~\ref{Lemma6} there are paths $\bbp_1', \bbp_2' \in P_+(\Gamma)$, $\bbp_1'', \bbp_2'' \in P_-(\Gamma)$, and $\bbq_1, \bbq_2 \in \overline{B^1 \cdot P(\Gamma_G) \cdot C^1}$ such that $\bbp_i \sim_X \bbp_i' \bbq_i \bbp_i''$ for $i=1,2$. By Corollary~\ref{Corollary} there are paths $\bbq_3, \bbq_4 \in \overline{B^1 \cdot \mathcal{P}(\Gamma_G) \cdot C^1}$ such that $\bbp_1' \bbq_3 \sim_X \bbp_2'$, and $\bbp_2'' \sim_X \bbq_4 \bbp_1''$.

\

\noindent \textbf{Claim.}
\noindent \textit{$\bbq_1 \sim_{X} \bbq_3 \bbq_2 \bbq_4$.}
\begin{proof}[Proof of Claim]
First we show that for any pair of parallel paths $(\bbq, \bbq')$ with $\bbq, \bbq'$ both belonging to $B^1 \cdot P(\Gamma_G) \cdot C^1$,  we have $\bbq \sim_{X_G} \bbq'$. Indeed, in this situation by Lemma~\ref{dagger} there exist $b \in B^1$ and $c \in C^1$ with $\bbq, \bbq' \in b \cdot \Gamma_G \cdot c$. Hence we have $\bbq = b \cdot \bbr \cdot c$ for some $\bbr \in \Gamma_G$, and $\bbq' = b \cdot \bbr' \cdot c$ for some $\bbr' \in \Gamma_G$. Since $X_G$ is a homotopy base for $\Gamma_G$ it follows that $\bbr \sim_{X_G} \bbr' $ which, by acting on the left by $b$ and the right by $c$, implies that $\bbq \sim_{X_G} \bbq'$. 

Returning to the proof of the claim, suppose that neither $\iota \bbq_1$ nor $\tau \bbq_1$ belongs to $B \cup C \cup BC$. Then by Lemma~\ref{newlemma}(i) it follows that there exist paths $\bbq_5, \bbq_6$ in $B^1 \cdot P(\Gamma_G) \cdot C^1$ such that $\bbq_1 \sim_X \bbq_5$ and $\bbq_3 \bbq_2 \bbq_4 \sim_X \bbq_6$. By the previous paragraph, since $\bbq_5, \bbq_6 \in B^1 \cdot P(\Gamma_G) \cdot C^1$ are parallel paths we obtain $\bbq_5 \sim_{X} \bbq_6$ and hence
\[
\bbq_1 \sim_X \bbq_5 \sim_{X} \bbq_6 \sim_X \bbq_3 \bbq_2 \bbq_4. 
\] 
Now consider the case that one of $\iota \bbq_1$ or $\tau \bbq_1$ belongs to $B \cup C \cup BC$. Suppose the former, the latter case is dealt with in much the same way. Now let $\bbq'' = \bbq_1 \bbq_4^{-1} \bbq_2^{-1} \bbq_3^{-1}$, so $\bbq''$ is a closed path with $\iota \bbq'' \equiv \tau \bbq'' \in B \cup C \cup BC$. By Lemma~\ref{newlemma}(ii) there exists a path $\bbq'''$ in $B^1 \cdot P(\Gamma_G) \cdot C^1$ and edges $\bbe_1$, $\bbe_2$ of $\overline{B^1 \cdot P(\Gamma_G) \cdot C^1}$ such that $\bbq'' \sim_X \bbe_1 \bbq''' \bbe_2$ and $\bbe_1 \sim_X \bbe_2^{-1}$. Since $(1_{\iota \bbq'''},\bbq''')$ is a pair of parallel paths both from $B^1 \cdot P(\Gamma_G) \cdot C^1$ it follows from the observation above that 
$\bbq''' \sim_{X_G} 1_{\iota \bbq''''}$. We conclude 
\[
\begin{array}{c}
\bbq_1 \sim \bbq_1 (\bbq_3 \bbq_2 \bbq_4)^{-1} (\bbq_3 \bbq_2 \bbq_4) =
\bbq'' (\bbq_3 \bbq_2 \bbq_4) \sim_X 
\bbe_1 \bbq''' \bbe_2 (\bbq_3 \bbq_2 \bbq_4)  \\
\sim_X \bbe_1 1_{\iota \bbq'''} \bbe_2 (\bbq_3 \bbq_2 \bbq_4) =
\bbe_1 \bbe_2 (\bbq_3 \bbq_2 \bbq_4) \sim_X
\bbe_2^{-1} \bbe_2 (\bbq_3 \bbq_2 \bbq_4) \sim
\bbq_3 \bbq_2 \bbq_4.
\end{array}
\]
This covers all possible cases and completes the proof of the claim.
\end{proof}

Using the claim we conclude 
\[
\bbp_1 \sim_X
\bbp_1' \bbq_1 \bbp_1'' \sim_X
\bbp_1' \bbq_3 \bbq_2 \bbq_4 \bbp_1'' \sim_X
\bbp_2' \bbq_2 \bbp_2'' \sim_X
\bbp_2,
\]
completing the proof in the case that $\iota \bbp_1$ represents an element of $S \setminus T$.
If $\iota \bbp_1$ does represent an element of $T$ then the proof is more straightforward. We follow the same steps as above (using part (ii) each time a lemma is applied) except that $\bbq_i \in P(\Gamma_T)$ for $i=1,2,3,4$, and we make use of the fact that $X_T$ is a homotopy base for $\Gamma_T$ to prove $\bbq_1 \sim_{X_T} \bbq_3 \bbq_2 \bbq_4$.
\end{proof}

To complete the proof of Theorem~\ref{comp0simpExt}, suppose that $G$ and $T$ have finite derivation type. Let $\mathcal{P}_G$ and $\mathcal{P}_T$ be finite presentations for $G$ and $T$ respectively. Then $\mathcal{P}_S$ defined above is a finite presentation for $S$. Let $X_G$ and $X_T$ be finite homotopy bases for $\mathcal{P}_G$ and $\mathcal{P}_T$ respectively. Then we have seen above that $X = X_1 \cup X_1' \cup X_2 \cup X_3 \cup X_e \cup X_G \cup X_T$ is a finite homotopy base for $\mathcal{P}_S$. Hence $S$ has finite derivation type. $\qed$

\end{document}